\documentclass[a4paper, 9pt]{article}
\usepackage{graphics}
\usepackage[dvips]{color}

\usepackage{lscape}
\usepackage{latexsym}
\usepackage{amsmath,verbatim}
\usepackage{graphics}
\usepackage{amsthm}
\usepackage{amssymb}
\usepackage{makeidx}
\usepackage{stmaryrd}
\usepackage{multicol}
\usepackage{helvet}
\usepackage{bm}
\usepackage[all]{xy}
\newfont{\Fr}{eufm10}
\newfont{\Sc}{eusm10}
\newfont{\Bb}{msbm10}
\newfont{\Am}{msam10}
\newfont{\am}{msam7}
\numberwithin{equation}{section}
\newtheorem{theorem}{Theorem}[section]
\newtheorem{proposition}[theorem]{Proposition}
\newtheorem{lemma}[theorem]{Lemma}
\newtheorem{corollary}[theorem]{Corollary}
\newtheorem{claim}{Claim}{\bf}{\it}

\newtheorem{ftheorem}{Theorem}{\bf}{\it}
{\bf}{\it}
\theoremstyle{definition}
\newtheorem{definition}[theorem]{Definition}

{\bf}{\rm}

\theoremstyle{remark}
\newtheorem{example}[theorem]{Example}
\newtheorem{remark}[theorem]{Remark}
{\bf}{\it}

\newtheorem{definition and corollary}[theorem]{Definition and Corollary}

\newtheorem{fexample}[ftheorem]{Example}{\it}{\rm}

\newcommand{\h}{\mathfrac{\h}}

\title{PBW bases and KLR algebras\footnote{Correction abridged version. Originally published in {\it Duke Mathematical Journal} 163 no.3 619--663 (2014)}}
\author{Syu \textsc{Kato} \footnote{Department of Mathematics, Kyoto University, Oiwake Kita-Shirakawa Sakyo Kyoto 606-8502, Japan. \tt{E-mail:syuchan@math.kyoto-u.ac.jp}} \footnote{Research supported in part by JSPS Grant-in-Aid for Young Scientists (B) 23-740014.}}

\begin{document}
\maketitle


\begin{abstract}
We generalize Lusztig's geometric construction of the PBW bases of finite quantum groups of type $\mathsf{ADE}$ under the framework of [Varagnolo-Vasserot, J. reine angew. Math. 659 (2011)]. In particular, every PBW basis of such quantum groups is proven to yield a semi-orthogonal collection in the module category of the KLR-algebras. This enables us to prove Lusztig's conjecture on the positivity of the canonical (lower global) bases in terms of the (lower) PBW bases. In addition, we verify Kashiwara's problem on the finiteness of the global dimensions of the KLR-algebras of type $\mathsf{ADE}$.
\end{abstract}

\section*{Introduction}
Canonical/global bases of quantum groups, defined by Lusztig \cite{Lu} and Kashiwara \cite{Ka} subsequently, open up scenery in many areas of mathematics which are visible only through quantum groups \cite{A2,Lu5,Na}. They are certain bases of quantum groups different from the natural quantum analogue of the classical Poincar\'e-Birkhoff-Witt theorem (that are usually referred to as the PBW bases).

Among these, the interaction between canonical/global bases of quantum groups and affine Hecke algebras of type $\mathsf{A}$ (and their cyclotomic quotients) yields many representation-theoretic consequences \cite{A, A2}. It is generalized to more general quantum groups and their representations by Khovanov-Lauda, Rouquier, Varagnolo-Vasserot, Zheng, Webster, and Kang-Kashiwara \cite{KL,R,VV,Z,W, KK} as a categorical counter-part of the theory of canonical/global bases.

More precisely, to each symmetric Kac-Moody algebra $\mathfrak g$, they introduced a series of algebras $R_{\beta}$ (that we call the KLR-algebras) whose simple/projective modules give rise to the upper/lower global bases of the corresponding positive half of the quantum group of $\mathfrak g$. There the emphasis is on the categorification of quantum groups, and their results are strong enough to generalize and categorify Ariki's result \cite{A} in these cases (Lauda-Vazirani \cite{LV} and \cite{VV,KK}).

This story is sufficient to recover deep representation-theoretic properties, without the PBW bases. The main observation of this paper is that the PBW bases still exist in the world of KLR-algebras, with essential new features which are visible only with the KLR-algebras.

To see what we mean by this, we prepare some notations: Let $\mathcal A := \mathbb Z [t ^{\pm 1}]$. Let $\mathfrak g$ be a simple Lie algebra of type $\mathsf{ADE}$, and let $U ^+$ be the positive half of the $\mathcal A$-integral version of the quantum group of $\mathfrak g$ (see e.g. Lusztig \cite{QG} \S 1). Let $Q ^+ := \mathbb Z_{\ge 0} I$, where $I$ is the set of positive simple roots. We have a weight space decomposition $U^+ = \bigoplus _{\beta \in Q ^+} U^+ _{\beta}$. We have the Weyl group $W$ of $\mathfrak g$ with its set of simple reflections $\{ s_i \} _{i \in I}$ and the longest element $w_0$. For each $\beta \in Q^+$, we have a finite set $B ( \infty )_{\beta}$ which parameterizes a pair of distinguished bases $\{ G ^{up} ( b ) \}_{b \in B ( \infty )_{\beta}}$ and $\{ G ^{low} ( b ) \} _{b \in B ( \infty )_{\beta}}$ of $\mathbb Q ( t ) \otimes _{\mathcal A} U^+ _{\beta}$. The Khovanov-Lauda-Rouquier algebra $R_{\beta}$ is a certain graded algebra whose grading is bounded from below with the following properties:
\begin{itemize}
\item The set of isomorphism classes of simple graded $R _{\beta}$-modules (up to grading shifts) is also parameterized by $B ( \infty )_{\beta}$;
\item For each $b \in B ( \infty )_{\beta}$, we have a simple graded $R_{\beta}$-module $L_b$ and its projective cover $P_b$. Let $L_{b'} \left< k \right>$ be the grade $k$ shift of $L_{b'}$, and let $[P_b : L_{b'} \left< k \right>]_0$ be the multiplicity of $L_{b'} \left< k \right>$ in $P_b$ (that is finite). Then, we have
$$G ^{low} ( b ) = \sum_{b' \in B ( \infty )_{\beta}, k \in \mathbb Z} t^k [P_b : L_{b'} \left< k \right>]_0 G^{up} ( b' );$$
\item For each $\beta, \beta' \in Q^+$, there exists an induction functor
$$\star : R_{\beta} \mathchar`-\mathsf{gmod} \times R_{\beta'} \mathchar`-\mathsf{gmod} \ni (M,N) \mapsto M \star N \in R_{\beta + \beta'} \mathchar`-\mathsf{gmod};$$
\item $\mathbf K := \bigoplus _{\beta \in Q^+} \mathbb Q ( t ) \otimes _{\mathcal A} K ( R_{\beta} \mathchar`-\mathsf{gmod} )$ is an associative algebra isomorphic to $\mathbb Q ( t ) \otimes _{\mathcal A} U^+$ with its product inherited from $\star$ (and the $t$-action is a grading shift).
\end{itemize}

As mentioned earlier, Lusztig \cite{Lu} studied the geometric side of the story. By applying the results in \cite{K5}, we first observe the following:

\begin{ftheorem}[Kashiwara's problem $=$ Corollary \ref{Kashiwara}]\label{fKashiwara}
For every $\beta \in Q ^+$, the algebra $R_{\beta}$ has finite global dimension.
\end{ftheorem}

This problem is raised by Kashiwara several times in his lectures on KLR algebras. We remark that in type $\mathsf{A}$ case, Theorem \ref{fKashiwara} follows through a Morita equivalence with an affine Hecke algebra of type $\mathsf{A}$ (see e.g. Opdam-Solleveld \cite{OS}).

For quantum groups, a way to construct a (nice) PBW basis depends on an arbitrary sequence $\mathbf i := (i_1, i_2, \cdots, i_{\ell}) \in I ^{\ell}$ corresponding to a reduced expression of $w_0$. Associated to $\mathbf i$, we have a total order $< _{\mathbf i}$ on each $B ( \infty ) _{\beta}$ (see \S \ref{Chap:charPBW}). We define two collections of graded $R_{\beta}$-modules $\{ \widetilde{E} _b ^{\mathbf i} \} _{b \in B ( \infty )_{\beta}}$ and $\{ E _b ^{\mathbf i} \} _{b \in B ( \infty )_{\beta}}$ as follows (cf. Corollary \ref{PBW-Kos}): {\bf 1)} $\widetilde{E} _b ^{\mathbf i}$ is obtained from $P_b$ by annihilating all $L_{b'} \left< k \right>$ with $b' < _{\mathbf i} b$ and $k \ge 0$, and {\bf 2)} $E _b ^{\mathbf i}$ is obtained from $\widetilde{E} _b ^{\mathbf i}$ by annihilating all $L_b \left< k \right>$ with $k > 0$.

Since $R_{\beta}$ is a graded algebra with finite global dimension, we set
$$\left< M, N \right> _{\mathsf{gEP}} := \sum _{i \ge 0} (-1)^i \mathsf{gdim} \, \mathrm{ext} ^i _{R_{\beta}} ( M, N ) \in \mathbb Q (t) \hskip 1mm \text{ for } \hskip 1mm M, N \in R _{\beta} \mathchar`-\mathsf{gmod},$$
where $\mathrm{hom}_{R _{\beta}} ( M, N ) := \bigoplus _{k \in \mathbb Z} \mathrm{Hom} _{R_{\beta} \mathchar`-\mathsf{gmod}} ( M \left< k \right>, N )$.

By construction, we deduce that the graded character expansion coefficient $[M : \widetilde{E} ^{\mathbf i}_{b'}] \in \mathbb Z (\!( t )\!)$ is well-defined for every $M \in R_{\beta} \mathchar`-\mathsf{gmod}$.

The above definitions of $\widetilde{E} ^{\mathbf i} _b$ and $E ^{\mathbf i} _b$ look natural, but not apparently related to a PBW basis of $U ^+$.

\begin{ftheorem}[Orthogonality relation $=$ Theorem \ref{main} and its corollaries]\label{forth} In the above setting, we have:
\begin{enumerate}
\item For $b < _{\mathbf i} b'$, we have $\mathrm{ext} ^{\bullet} _{R_{\beta}} ( E _b ^{\mathbf i}, E _{b'} ^{\mathbf i} ) = \{ 0 \}$;
\item We have
$$\mathrm{ext} ^{i} _{R_{\beta}} (  \widetilde{E} ^{\mathbf i} _{b}, ( E ^{\mathbf i} _{b'} ) ^* ) = \begin{cases} \mathbb C & (b \neq b', i = 0)\\ \{ 0 \} & (otherwise)\end{cases}, \hskip 2mm \text{ and } \left< \widetilde{E} _b ^{\mathbf i}, ( E _{b'} ^{\mathbf i} ) ^* \right> _{\mathsf{gEP}} = \delta _{b,b'};$$
\item The graded $R_{\beta}$-module $\widetilde{E} _b ^{\mathbf i}$ is a self-extension of $E _{b} ^{\mathbf i}$ in the sense that there exists a separable decreasing filtration of $\widetilde{E} _b ^{\mathbf i}$ whose associated graded is a direct sum of grading shifts of $E _{b} ^{\mathbf i}$.
\end{enumerate}
\end{ftheorem}

Since we have $\left< P_b, L_{b'} \right>_{\mathsf{gEP}} = \delta_{b,b'}$ by definition, the pairing $\left< \bullet, \bullet \right>_{\mathsf{gEP}}$ is essentially the Lusztig inner form (cf. \cite{QG} 1.2.10--1.2.11). Therefore, Theorem \ref{forth} guarantees that our $\{ \widetilde{E} _b ^{\mathbf i} \} _{b}$, and $\{ E _{b} ^{\mathbf i} \} _b$ must be the categorifications of the lower/upper PBW bases by their characterization. We remark that some of these modules seem to coincide with those obtained by Kleshchev-Ram \cite{KR}, Webster \cite{W}, and Benkart-Kang-Oh-Park \cite{BKOP}.

\begin{ftheorem}[Lusztig's conjecture $=$ Theorem \ref{Lusztig}]\label{fLusztig}
We have $[P_b : \widetilde{E} ^{\mathbf i}_{b'}] = [ E ^{\mathbf i}_{b'} : L_b ]$ for each $b,b' \in B (\infty) _{\beta}$. In particular, we have
$$[P_b : \widetilde{E} ^{\mathbf i}_{b'}] \in \mathbb N [t] \hskip 3mm \text{ for every } \hskip 3mm b,b' \in B (\infty) _{\beta}.$$
\end{ftheorem}

Theorem \ref{fLusztig} is conjectured by Lusztig as his comment on \cite{Lu} in his webpage. Note that Theorem \ref{fLusztig} is established in Lusztig \cite{Lu} Corollary 10.7 when the reduced expression $\mathbf i$ satisfies the condition so-called ``adapted" (see \S 3).

\begin{fexample}[$\mathfrak g = \mathfrak{sl}_3$]
We have $I = \{ \alpha _1, \alpha _2 \}$. The standard generators $E_1, E_2$ of $U^+$ correspond to projective modules $P_1$ and $P_2$ of $R_{\alpha _1}$ and $R_{\alpha_2}$, respectively. Then, one series of the (lower) PBW basis $\{ \widetilde{E} ^{\mathbf i} _{b} \} _b$ are:
$$P_1 ^{(c_1)} \star Q_{21}^{(c_2)} \star P_2 ^{(c_3)} \hskip 3mm \text{ for } \hskip 3mm c_1,c_2,c_3 \ge 0.$$
Here $X ^{(c)}$ denotes a direct factor of $X \star X \star \cdots \star X$ ($c$ times). Note that $P_1 ^{(c_1)}$, $Q_{21} ^{(c_2)}$, and $P_2 ^{(c_3)}$ are maximal self-extensions of simple modules (this is a general phenomenon). We have a short exact sequence
$$0 \to P _1 \star P _2 \left< 2 \right> \longrightarrow P_2 \star P_1 \longrightarrow Q_{21} \to 0,$$
which is a categorical version of $E_2 E_1 - t^2 E_1 E_2$.
\end{fexample}

The organization of this paper is as follows: In the first section, we collect several results from \cite{K5} needed in the sequel. The second section is the preliminary on the KLR algebra. In the third section, we abstract and categorify Lusztig's arguments in the setting of the Hall algebras \cite{Lu2} to the KLR algebras by utilizing the results of \cite{K5} and the induction theorem imported from \cite{KaL,L-CG3,K1}. This includes categorifications of Saito's reflection actions \cite{S} that we call the Saito reflection functors. In the fourth section, we depart from geometry and utilize the properties of the Saito reflection functors established in the earlier sections to deduce Theorem \ref{forth} and Theorem \ref{fLusztig}.

After submitted the initial version this paper, there appeared another (algebraic) proofs of the main results of this paper by McNamara \cite{M} and Brundan-Kleshchev-McNamara \cite{BKM}, which also covers non-simply laced cases. Thier approach is quite different from that of ours, and one merit of our approach is that it provides a bridge between geometric/algebraic view points, typically seen in the Saito reflection functor used in the proof.

\begin{flushleft}
{\small
{\bf Acknowledgement:} The author is indebted to Masaki Kashiwara for helpful discussions, comments, and lectures on this topic. He is also indebted to Yoshiyuki Kimura for helpful discussions, comments, and pointing out some errors. He is also grateful to Michela Varagnolo for pointing out a reference.\\
This version corrects couple of errors in the published version of this paper, most notably these in Lemma 4.2 2) noticed by Euiyong Park and communicated to the author by Myungho Kim. The author wants to express his thanks to them.}
\end{flushleft}

\begin{flushleft}
{\small {\bf Convention}}
\end{flushleft}

An algebra $R$ is a (not necessarily commutative) unital $\mathbb C$-algebra. A variety $\mathfrak X$ is a separated reduced scheme $\mathfrak X_0$ of finite type over some localization $\mathbb Z_S$ of $\mathbb Z$ specialized to $\mathbb C$. It is called a $G$-variety if we have an action of a connected affine algebraic group scheme $G$ flat over $\mathbb Z_S$ on $\mathfrak X_0$ (specialized to $\mathbb C$). As in \cite{BBD} \S 6 and \cite{BL}, we transplant the notion of weights to the derived category of ($G$-equivariant) constructible sheaves with finite monodromy on $\mathfrak X$. Let us denote by $D^b ( \mathfrak X )$ (resp. $D^+ ( \mathfrak X )$) the bounded (resp. bounded from the below) derived category of the category of constructible sheaves on $\mathfrak X$, and denote by $D^+ _G ( \mathfrak X )$ the $G$-equivariant derived category of $\mathfrak X$. We have a natural forgetful functor $D^+ _G ( \mathfrak X ) \to D^+ ( \mathfrak X )$, whose preimage of $D^b ( \mathfrak X )$ is denoted by $D ^b _G ( \mathfrak X )$. For an object of $D^b_G ( \mathfrak X )$, we may denote its image in $D ^b ( \mathfrak X )$ by the same letter.

Let $\mathsf{vec}$ be the category of $\mathbb Z$-graded vector spaces (over $\mathbb C$) bounded from the below so that its objects have finite-dimensional graded pieces. In particular, for $V = \oplus _{i \gg - \infty} V^i \in \mathsf{vec}$, its graded dimension $\mathsf{gdim} \, V := \sum _{i} t ^i \dim V_i \in \mathbb Z (\!( t )\!)$ makes sense (with $t$ being indeterminant). We define $V \left<m\right>$ by setting $( V \left<m\right> ) _i := V _{i - m}$.

In this paper, a graded algebra $A$ is always a $\mathbb C$-algebra whose underlying space is in $\mathsf{vec}$. Let $A \mathchar`-\mathsf{gmod}$ be the category of finitely generated graded $A$-modules. For $E, F \in A \mathchar`-\mathsf{gmod}$, we define $\hom_{A} ( E, F )$ to be the direct sum of graded $A$-module homomorphisms $\hom _{A} ( E, F )^j$ of degree $j$ ($= \mathrm{Hom} _{A \mathchar`-\mathsf{gmod}} ( E \left< j \right>, F )$). We employ the same notation for extensions (i.e. $\mathrm{ext}_{A} ^i ( E, F ) = \oplus_{j\in \mathbb Z} \mathrm{ext}^{i}_{A} (E,F)^j$). We denote by $\mathsf{Irr} \, A$ the set of isomorphism classes of graded simple modules of $A$, and denote by $\mathsf{Irr} _0 \, A$ the set of isomorphism classes of graded simple modules of $A$ up to grading shifts. Two graded algebras are said to be Morita equivalent if their graded module categories are equivalent. For a graded $A$-module $E$, we denote its head by $\mathsf{hd} \, E$, and its socle by $\mathsf{soc} \, E$.

For $Q (t) \in \mathbb Q (t)$, we set $\overline{Q} ( t ) := Q ( t^{-1} )$. For derived functors $\mathbb R F$ or $\mathbb L F$ of some functor $F$, we represent its arbitrary graded piece (of its homology complex) by $\mathbb R ^{*} F$ or $\mathbb L ^{*} F$, and the direct sum of whole graded pieces by $\mathbb R ^{\bullet} F$ or $\mathbb L ^{\bullet} F$. For example, $\mathbb R ^{*} F \cong \mathbb R ^{*} G$ means that $\mathbb R ^{i} F \cong \mathbb R ^{i} G$ for every $i \in \mathbb Z$, while $\mathbb R ^{\bullet} F \cong \mathbb R ^{\bullet} G$ means that $\bigoplus _{i} \mathbb R ^{i} F \cong \bigoplus _{i} \mathbb R ^{i} G$.

When working on some sort of derived category, we suppress $\mathbb R$ or $\mathbb L$, or the category from the notation for simplicity when there is only small risk of confusion.

\section{Recollection from \cite{K5}}\label{general_Ext}
Let $G$ be a connected reductive algebraic group. Let $\mathfrak X$ be a $G$-variety. Let $\Lambda$ be the labelling set of $G$-orbits of $\mathfrak X$. For $\lambda \in \Lambda$, we denote the corresponding $G$-orbit by $\mathbb O _{\lambda}$. For $\lambda, \mu \in \Lambda$, we write $\lambda \preceq \mu$ if $\mathbb O _{\lambda} \subset \overline{\mathbb O _{\mu}}$. We assume the following property $(\spadesuit)$:

\begin{itemize}
\item[$(\spadesuit)_1$] The set $\Lambda$ is finite. For each $\lambda \in \Lambda$, we fix $x_{\lambda} \in \mathbb O _{\lambda} ( \mathbb C )$;
\item[$(\spadesuit)_2$] For each $\lambda \in \Lambda$, the group $\mathsf{Stab} _G ( x _{\lambda} )$ is connected.
\end{itemize}

We have a (relative) dualizing complex $\omega_{\mathfrak X} := p^! \underline{\mathbb C} \in D^b _{G} ( \mathfrak X )$, where $p : \mathfrak X \to \{\mathrm{pt}\}$ is the $G$-equivariant structure map. We have a dualizing functor
$$\mathbb D : D ^b _G ( \mathfrak X ) ^{op} \ni C^{\bullet} \mapsto \mathcal{H}om ^{\bullet} ( C ^{\bullet}, \omega_{\mathfrak X} ) \in D^b _G ( \mathfrak X ).$$

We have a $\mathbb D$-autodual $t$-structure of $D^b _G ( \mathfrak X )$ whose truncation functor and perverse cohomology functor are denoted by $\tau$ and ${}^p H$, respectively.

For each $\lambda \in \Lambda$, we have a constant local system $\underline{\mathbb C} _{\lambda}$ on $\mathbb O _{\lambda}$. We have inclusions $i_{\lambda} : \{ x_{\lambda} \} \hookrightarrow \mathfrak X$ and $j _{\lambda} : \mathbb O_{\lambda} \hookrightarrow \mathfrak X$. Let $\mathbb C _{\lambda} := ( j _{\lambda} ) _! \underline{\mathbb C} _{\lambda} [ \dim \mathbb O _{\lambda} ]$ and $\mathsf{IC} _{\lambda} := ( j _{\lambda} ) _{!*} \underline{\mathbb C} _{\lambda} [ \dim \mathbb O _{\lambda} ]$, which we regard as objects of $D^b _G ( \mathfrak X )$. We denote by
\begin{align*}
\mathrm{Ext} ^{\bullet} _G ( \bullet, \bullet ) & : D^b _G ( \mathfrak X ) ^{op} \times D^b _G ( \mathfrak X ) \longrightarrow D^+ ( \mathrm{pt} )\\
\mathrm{Ext} ^{\bullet} ( \bullet, \bullet ) & : D ^b ( \mathfrak X ) ^{op} \times D ^b ( \mathfrak X ) \longrightarrow D ^b ( \mathrm{pt} )
\end{align*}
the Ext (as bifunctors) of $D^b_G ( \mathfrak X )$ and $D^b ( \mathfrak X )$, respectively.

For each $\lambda \in \Lambda$, we fix $L _{\lambda} \in D ^b ( \mathrm{pt} )$ as a non-zero graded vector space with a trivial differential which satisfies the self-duality condition $L _{\lambda} \cong L _{\lambda} ^*$. We set
$$\mathcal L := \bigoplus _{\lambda \in \Lambda} L_{\lambda} \boxtimes \mathsf{IC} _{\lambda} \in D _G ^b ( \mathfrak X ).$$
By construction, we find an isomorphism $\mathcal L \cong \mathbb D \mathcal L$.

We form a graded Yoneda algebra
$$A _{(G, \mathfrak X)} = \bigoplus _{i \in \mathbb Z} A ^i _{(G, \mathfrak X)} := \bigoplus _{i \in \mathbb Z} \mathrm{Ext} _G ^i ( \mathcal L, \mathcal L )$$
whose degree is the cohomological degree. We denote by $B _{(G, \mathfrak X)}$ the algebra $A_{(G, \mathfrak X)}$ by taking $\mathcal L = \bigoplus _{\lambda \in \Lambda} \mathsf{IC} _{\lambda}$ (and call it the basic ring of $A_{(G, \mathfrak X)}$). The algebra $B_{(G, \mathfrak X)}$ is Morita equivalent to $A _{(G, \mathfrak X)}$, and hence all the arguments in the below are independent of the choice of $\mathcal L$, which we suppress for simplicity. We also drop $(G, \mathfrak X)$ in case the meaning is clear from the context. It is standard that $\{L_{\lambda}\}_{\lambda \in \Lambda}$ forms a complete collection of graded simple $A$-modules up to grading shifts.

\begin{lemma}[see \cite{K5} 1.2]\label{symm}
For a graded $A$-module $M$, its graded dual $M^*$ is again a graded $A$-module. \hfill $\Box$
\end{lemma}

For each $\lambda \in \Lambda$, we set
$$P_{\lambda} := \mathrm{Ext} _G ^{\bullet} ( \mathsf{IC} _{\lambda}, \mathcal L ) = \bigoplus _{i \in \mathbb Z} \mathrm{Ext} _G ^i ( \mathsf{IC} _{\lambda}, \mathcal L ).$$
Each $P_{\lambda}$ is a graded projective left $A$-module. By construction, we have
$$A \cong \bigoplus _{\lambda \in \Lambda} L _{\lambda} ^* \boxtimes \mathrm{Ext} _G ^{\bullet} ( \mathsf{IC} _{\lambda}, \mathcal L ) = \bigoplus _{\lambda \in \Lambda} P _{\lambda} \boxtimes L _{\lambda} ^*$$
as left $A$-modules. It is standard that $P_{\lambda}$ is an indecomposable $A$-module whose head is isomorphic to $L _{\lambda}$ (cf. \cite{CG} \S 8.7). We have an idempotent $e_{\lambda} \in A$ so that $P_{\lambda} \cong A e _{\lambda}$ as left graded $A$-modules (up to a grading shift).

For each $\lambda \in \Lambda$, we set
$$\widetilde{K} _{\lambda} := \mathrm{Ext} _G ^{\bullet} ( \mathbb C _{\lambda}, \mathcal L ) \text{ and } K _{\lambda} := H ^{\bullet} i_{\lambda} ^! \mathcal L[ \dim \mathbb O _{\lambda} ].$$
We call $K _{\lambda}$ a standard module, and $\widetilde{K} _{\lambda}$ a dual standard module of $A$.

We regard each $\mathsf{IC} _{\lambda}$ as a simple mixed perverse sheaf (of weight zero) in the category of mixed sheaves on $\mathfrak X$ via \cite{BBD} \S 5 and \S 6, and each $L_{\lambda}$ as a mixed (complex of) vector space of weight zero. I.e. each $L_{\lambda} ^i$ is pure of weight $i$ in the sense that the geometric Frobenius acts by $q ^{i/2} \mathsf{id}$ if we switch the base field to the algebraic closure of a finite field of cardinality $q$. It follows that the algebra $A$ acquires a (mixed) weight structure.

We consider the following property $(\clubsuit)$:
\begin{enumerate}
\item[$(\clubsuit)_1$] The algebra $A$ is pure of weight $0$;
\item[$(\clubsuit)_2$] For each $\lambda \in \Lambda$, the perverse sheaf $\mathsf{IC} _{\lambda}$ is pointwise pure;
\end{enumerate}

\begin{theorem}[\cite{K5} 3.5]\label{absKas} Assume the properties $(\spadesuit)$ and $(\clubsuit)$. Then, the algebra $A$ has finite global dimension. \hfill $\Box$
\end{theorem}

For $M \in A \mathchar`-\mathsf{gmod}$ and $i \in \mathbb Z$, we define
\begin{align*}
& [ M : L_{\lambda} \left< i \right> ]_0 := \dim \, \mathrm{Hom} _{A\mathchar`-\mathsf{gmod}} ( P_{\lambda} \left< i \right>, M ) \in \mathbb Z \hskip 2mm \text{ and }\\
& [ M : L_{\lambda} ] := \mathsf{gdim} \, \mathrm{hom} _A ( P_{\lambda}, M ) \in \mathbb Z (\!(t)\!).
\end{align*}

We have $[M : L_{\lambda}] = \sum _{i \in \mathbb Z} [ M : L_{\lambda} \left< i \right> ]_0 t^i \in \mathbb Z (\!(t)\!)$.

\begin{theorem}[\cite{K5} 1.6]\label{KS} Assume the properties $(\spadesuit)$ and $(\clubsuit)$:
\begin{enumerate}
\item We have
$$[\widetilde{K} _{\lambda} : L _{\mu} ] = 0 = [K _{\lambda} : L _{\mu} ] \hskip 2mm \text{ for } \lambda \not\preceq \mu \hskip 2mm \text{ and } \hskip 3mm [K _{\lambda} : L _{\lambda} ] = 1;$$
\item For each $\mu \not\preceq \lambda$, we have
$$\mathrm{ext} _A ^{\bullet} ( \widetilde{K} _{\lambda}, \widetilde{K} _{\mu} ) = \{ 0 \} \hskip 3mm \text{ and } \hskip 3mm \mathrm{ext} _A ^{\bullet} ( K _{\lambda}, K _{\mu} ) = \{ 0 \};$$
\item For each $\lambda \in \Lambda$, we have
$$\widetilde{K} _{\lambda} \cong P _{\lambda} / \Bigl(  \sum _{\mu \prec \lambda} A e_{\mu} P _{\lambda} \Bigr);$$
\item Each $\widetilde{K} _{\lambda}$ is a successive self-extension of $K_{\lambda}$. In addition, we have
$$[\widetilde{K} _{\lambda} : L _{\lambda} ] = \mathsf{gdim} \, H ^{\bullet} _{\mathsf{Stab} _G ( x_{\lambda} )} ( \mathrm{pt} ).$$
\end{enumerate}
\end{theorem}

For $M \in A \mathchar`-\mathsf{gmod}$ and $N \in A \mathchar`-\mathsf{gmod}$, we define its graded Euler-Poincar\'e characteristic as:
\begin{equation}
\left< M, N \right> _{\mathsf{gEP}} := \sum _{i \ge 0} (-1)^i \mathsf{gdim} \, \mathrm{ext} ^i _A ( M, N ) \in \mathbb Z (\!( t )\!).\label{gepdef}
\end{equation}

Let $j : \mathfrak Y \hookrightarrow \mathfrak X$ be the inclusion of an open $G$-stable subvariety. We form a graded algebra
$$A_{(G, \mathfrak Y)} := \mathrm{Ext} ^{\bullet} _{G} ( j ^* \mathcal L, j ^* \mathcal L ).$$

\begin{lemma}[\cite{K5} 4.4]\label{localKsc}
Let $j : \mathfrak Y \hookrightarrow \mathfrak X$ be the inclusion of an open $G$-stable subvariety. Then, $\mathfrak Y$ satisfies the conditions $(\spadesuit)$ and $(\clubsuit)$ if $\mathfrak X$ does. \hfill $\Box$
\end{lemma}

\begin{proposition}[\cite{K5} 4.3, 4.5]\label{Kquotients}
Let $i : \mathbb O_{\lambda} \hookrightarrow \mathfrak X$ be the inclusion of a closed $G$-orbit $($with $\lambda \in \Lambda)$, and let $j : \mathfrak Y \hookrightarrow \mathfrak X$ be its complement. Then, we have an isomorphism $A_{(G, \mathfrak X)} / ( A_{(G, \mathfrak X)}e_{\lambda}A_{(G, \mathfrak X)} ) \stackrel{\cong}{\longrightarrow} A _{(G, \mathfrak Y)}$. \hfill $\Box$
\end{proposition}

\begin{corollary}[\cite{K5} 4.2, 4.3, 4.5]\label{isomKext}
Let $j : \mathfrak Y \hookrightarrow \mathfrak X$ be the inclusion of an open $G$-stable subvariety. We have
$$\mathrm{ext} ^{*}_{A} ( A, L_{\mu} ) \cong \mathrm{ext} ^{*} _{A} ( A_{(G, \mathfrak Y)}, L_{\mu} )\label{isomext}$$
for every $\mu \in \Lambda$ so that $\mathbb O_{\mu} \subset \mathfrak Y$. \hfill $\Box$
\end{corollary}

\section{Quivers and the KLR algebras}
Let $\Gamma = (I, \Omega)$ be an oriented graph with the set of its vertex $I$ and the set of its oriented edges $\Omega$. Here $I$ is fixed, and $\Omega$ might change so that the underlying graph $\Gamma_0$ of $\Gamma$ is a fixed Dynkin diagram of type $\mathsf{ADE}$. We refer $\Omega$ as the orientation of $\Gamma$. We form a path algebra $\mathbb C [\Gamma]$ of $\Gamma$.

For $h \in \Omega$, we define $h' \in I$ to be the source of $h$ and $h'' \in I$ to be the sink of $h$. We denote $i \leftrightarrow j$ for $i, j \in I$ if and only if there exists $h \in \Omega$ such that $\{ h', h'' \} = \{ i, j \}$. A vertex $i \in I$ is called a sink of $\Gamma$ if $h' \neq i$ for every $h \in \Omega$. A vertex $i \in I$ is called a source of $\Gamma$ if $h'' \neq i$ for every $h \in \Omega$.

Let $Q^+$ be the free abelian semi-group generated by $\{ \alpha _i \} _{i \in I}$, and let $Q ^+ \subset Q$ be the free abelian group generated by $\{ \alpha _i \} _{i \in I}$. We sometimes identify $Q$ with the root lattice of type $\Gamma_0$ with a set of its simple roots $\{ \alpha _i \} _{i \in I}$. Let $W = W ( \Gamma_0 )$ denote the Weyl group of type $\Gamma_0$ with a set of its simple reflections $\{ s_i \} _{i \in I}$. The group $W$ acts on $Q$ via the above identification. Let $R^+ := W \{ \alpha _i \} _{i \in I} \cap Q^+$ be the set of positive roots of a simple Lie algebra with its Dynkin diagram $\Gamma_0$.

An $I$-graded vector space $V$ is a vector space over $\mathbb C$ equipped with a direct sum decomposition $V = \bigoplus _{i \in I} V_i$.

Let $V$ be an $I$-graded vector space. For $\beta \in Q ^+$, we declare $\underline{\dim} \, V = \beta$ if and only if $\beta = \sum _{i \in I} ( \dim V _i ) \alpha_i$. We call $\underline{\dim} \, V$ the dimension vector of $V$. Form a vector space
$$E_{V} ^{\Omega} := \bigoplus _{h \in \Omega} \mathrm{Hom} _{\mathbb C} ( V _{h'}, V _{h''} ).$$
We set $G _V := \prod_{i \in I} \mathop{GL} ( V _{i} )$. The group $G_V$ acts on $E_V ^{\Omega}$ through its natural action on $V$. The space $E_V ^{\Omega}$ can be identified with the based space of $\mathbb C [\Gamma]$-modules with its dimension vector $\beta$. Let $\mathtt M _{i}$ be a unique $\mathbb C [\Gamma]$-module (up to an isomorphism) with $\underline{\dim} \, \mathtt M _{i} = \alpha _i$.

For each $k \ge 0$, we consider a sequence $\mathbf m = ( m_1,m_2,\ldots, m_{k} ) \in I ^{k}$. We abbreviate this as $\mathsf{ht} ( \mathbf m ) = k$. We set $\mathsf{wt} ( \mathbf m ) := \sum _{j = 1} ^k \alpha _{m_j} \in Q ^+$. For $\beta = \mathsf{wt} ( \mathbf m ) \in Q ^+$, we set $\mathsf{ht} \, \beta = k$. For a sequence $\mathbf m' := (m_1',\ldots, m'_{k'}) \in I ^{k'}$, we set
$$\mathbf m + \mathbf m' := ( m_1,\ldots, m_k, m_1',\ldots, m_{k'}') \in I ^{k+k'}.$$
For $i \in I$ and $k \ge 0$, we understand that $ki = (i,\ldots,i) \in I^k$.

For each $\beta \in Q ^+$, we set $Y ^{\beta}$ to be the set of all sequences $\mathbf m$ such that $\mathsf{wt} ( \mathbf m ) = \beta$. For each $\beta \in Q^+$ with $\mathsf{ht} \, \beta = n$ and $1 \le i < n$, we define an action of $\{\sigma_i\}_{i=1}^{n-1}$ on $Y^{\beta}$ as follows: For each $1 \le i < n$ and $\mathbf m = ( m_1, \ldots, m_{n} ) \in Y ^{\beta}$, we set
$$\sigma_i \mathbf m := ( m_1,\ldots, m_{i-1}, m_{i+1}, m_i, m_{i+2},\ldots, m_n).$$ 
It is clear that $\{\sigma_i\}_{i=1}^{n-1}$ generates a $\mathfrak S_n$-action on $Y^{\beta}$. In addition, $\mathfrak S_n$ naturally acts on a set of integers $\{ 1,2,\ldots, n \}$.

\begin{definition}[Khovanov-Lauda \cite{KL}, Rouquier \cite{R}]\label{KLR}
Let $\beta \in Q^+$ so that $n = \mathsf{ht} \, \beta$. We define the KLR algebra $R _{\beta}$ as a unital algebra generated by the elements $\kappa _1, \ldots, \kappa_n$, $\tau_1, \ldots, \tau_{n-1}$, and $e ( \mathbf m )$ $(\mathbf m \in Y ^{\beta})$ subject to the following relations:
\begin{enumerate}
\item $\deg \kappa _i e ( \mathbf m ) = 2$ for every $i$, and
$$\deg \tau _i e ( \mathbf m ) = \begin{cases}-2 & (m _{i} = m_{i+1}) \\1 & (m_i \leftrightarrow m_{i+1}) \\ 0 & (otherwise) \end{cases};$$
\item $[ \kappa _i, \kappa _j ] = 0$, $e ( \mathbf m ) e ( \mathbf m' ) = \delta _{\mathbf m, \mathbf m'} e ( \mathbf m )$, and $\sum _{\mathbf m \in Y ^{\beta}} e ( \mathbf m ) = 1$;
\item $\tau _i e ( \mathbf m ) = e ( \sigma _i \mathbf m ) \tau _i e ( \mathbf m )$, and $\tau _i \tau _j e ( \mathbf m ) = \tau _j \tau _i e ( \mathbf m )$ for $|i-j|>1$;
\item $\tau _i ^2 e ( \mathbf m ) = Q _{\mathbf m, i} ( \kappa _i, \kappa _{i+1} ) e ( \mathbf m )$;
\item For each $1 \le i < n$, we have
\begin{align*}
 \tau_{i+1} \tau_i \tau_{i+1} e ( \mathbf m ) - & \tau_i \tau_{i+1} \tau_i e ( \mathbf m )\\
& = \begin{cases} \frac{Q _{\mathbf m, i} ( \kappa _{i+2}, \kappa _{i+1} ) - Q _{\mathbf m, i} ( \kappa _i, \kappa _{i+1} )}{\kappa _{i+2} - \kappa _i} e ( \mathbf m )& (m_{i+2} = m_i) \\ 0 & (\text{otherwise}) \end{cases};
\end{align*}
\item $\tau_i \kappa _k e ( \mathbf m )- \kappa _{\sigma_i (k)} \tau_i e ( \mathbf m ) = \begin{cases} - e ( \mathbf m ) & (i=k, m_i = m_{i+1}) \\ e ( \mathbf m ) & (i=k-1, m_i = m_{i+1}) \\ 0 & (\text{otherwise})\end{cases}$.
\end{enumerate}
Here we set $h_{\mathbf m, i} := \# \{ h \in \Omega \mid h' = m_i, h'' = m_{i+1}\}$ and
$$Q _{\mathbf m, i} ( u,v ) = \begin{cases} 1 & (m_i \neq m _{i+1}, m_i \not\leftrightarrow m_{i+1}) \\
(-1)^{h_{\mathbf m, i}} ( u - v ) & (m_i \leftrightarrow m_{i+1}) \\ 0 & (\text{otherwise}) \end{cases},$$
where $u,v$ are indeterminants. \hfill $\Box$
\end{definition}

\begin{remark}
Note that the algebra $R_{\beta}$ a priori depends on the orientation $\Omega$ through $Q _{\mathbf m, i} ( u,v )$. Since the graded algebras $R_{\beta}$ are known to be mutually isomorphic for any two choices of $\Omega$ $($cf. Theorem \ref{VV}$)$, we suppress this dependence in the below.
\end{remark}

For an $I$-graded vector space $V$ with $\underline{\dim} \, V = \beta$, we define
\begin{align*}
F ^{\Omega} _{\beta} := & \Biggl\{ ( \{ F_j \} _{j=0} ^{\mathsf{ht} \beta }, x ) \Biggl| {\small \begin{matrix}  x \in E_{V} ^{\Omega}. \text{ For each $0 < j \le \mathsf{ht} \beta $,}\\
F _j \subset V \text{ is an $I$-graded vector subspace,}\\ F _{j+1} \subsetneq F_j \text{, and satisfies } x F _{j} \subset F _{j+1}. \end{matrix} }\Biggr\} \hskip 5mm \text{and}\\
\mathcal B ^{\Omega} _{\beta} := & \Biggl\{ \{ F_j \} _{j=0} ^{\mathsf{ht} \beta } \Biggl| {\small 
F _j \subset V \text{ is an $I$-graded vector subspace s.t. } F _{j+1} \subsetneq F_{j}. }\Biggr\}.
\end{align*}
We have a projection
$$\varpi _{\beta} ^{\Omega} : F ^{\Omega} _{\beta} \ni ( \{ F_j \} _{j=0} ^{\mathsf{ht} \beta }, x ) \mapsto \{ F_j \} _{j=0} ^{\mathsf{ht} \beta } \in \mathcal B ^{\Omega} _{\beta},$$
which is $G_V$-equivariant. For each $\mathbf m \in Y ^{\beta}$, we have a connected component
$$F ^{\Omega} _{\mathbf m} := \{( \{ F_j \} _{j=0} ^{\mathsf{ht} \beta }, x ) \in F ^{\Omega} _{\beta} \mid \underline{\dim} \, F_{j} / F_{j+1} = \alpha _{m_{j+1}} \hskip 2mm \forall j \} \subset F ^{\Omega} _{\beta},$$
that is smooth of dimension $d _{\mathbf m} ^{\Omega}$. We set $\mathcal B ^{\Omega} _{\mathbf m} := \varpi _{\beta} ^{\Omega} ( F ^{\Omega} _{\mathbf m} )$, that is an irreducible component of $\mathcal B ^{\Omega} _{\beta}$. Let
$$\pi ^{\Omega} _{\mathbf m} : F ^{\Omega} _{\mathbf m} \ni ( \{ F_j \} _{j=0} ^{\mathsf{ht} \beta }, x ) \mapsto x \in E ^{\Omega} _{V}$$
be the second projection that is also $G_V$-equivariant. The map $\pi ^{\Omega} _{\mathbf m}$ is projective, and hence
$$\mathcal L _{\mathbf m} ^{\Omega} := (\pi ^{\Omega} _{\mathbf m}) _! \, \underline{\mathbb C} \, [ d _{\mathbf m} ^{\Omega} ]$$
decomposes into a direct sum of (shifted) irreducible perverse sheaves with their coefficients in $D ^b ( \mathrm{pt} )$ (Gabber's decomposition theorem, \cite{BBD} 6.2.5). We set $\mathcal L ^{\Omega} _{\beta} := \bigoplus _{\mathbf m \in Y ^{\beta}} \mathcal L _{\mathbf m} ^{\Omega}$. Let $e ( \mathbf m )$ be the idempotent in $\mathrm{End} ( \mathcal L ^{\Omega} _{\beta} )$ so that $e ( \mathbf m )  \mathcal L ^{\Omega} _{\beta} =  \mathcal L ^{\Omega} _{\mathbf m}$. Since $\pi ^{\Omega} _{\mathbf m}$ is projective, we conclude that $\mathbb D \mathcal L _{\mathbf m} ^{\Omega} \cong \mathcal L _{\mathbf m} ^{\Omega}$ for each $\mathbf m \in Y ^{\beta}$, and hence
\begin{equation}
\mathbb D \mathcal L ^{\Omega} _{\beta} \cong \mathcal L ^{\Omega} _{\beta}. \label{L-auto}
\end{equation}

\begin{theorem}[Varagnolo-Vasserot \cite{VV}]\label{VV}
Under the above settings, we have an isomorphism of graded algebras:
$$R _{\beta} \cong \bigoplus _{i \in \mathbb Z} \mathrm{Ext} ^i _{G_V} ( \mathcal L ^{\Omega} _{\beta}, \mathcal L ^{\Omega} _{\beta} ).$$
In particular, the RHS does not depend on the choice of an orientation $\Omega$ of $\Gamma_0$.
\end{theorem}

For each $\mathbf m, \mathbf m' \in Y ^{\beta}$, we set
$$R _{\mathbf m, \mathbf m'} := e ( \mathbf m ) R _{\beta} e ( \mathbf m' ) =  \bigoplus _{i \in \mathbb Z} \mathrm{Ext} ^i _{G_V} ( \mathcal L ^{\Omega} _{\mathbf m'}, \mathcal L ^{\Omega} _{\mathbf m} ).$$

We set $S _{\beta} \subset R _{\beta}$ to be a subalgebra which is generated by $e ( \mathbf m )$ ($\mathbf m \in Y^{\beta}$) and $\kappa_1,\ldots, \kappa _n$.

For each $\beta_1, \beta _2 \in Q_+$ with $\mathsf{ht} \, \beta _1 = n_1$ and $\mathsf{ht} \, \beta _2 = n_2$, we have a natural inclusion:
$$
\xymatrix@=0pt{
& R _{\beta _1} \boxtimes R _{\beta _2} & \ni & e ( \mathbf m ) \boxtimes e ( \mathbf m' ) & \mapsto & e ( \mathbf m + \mathbf m' ) & \in & R _{\beta_1 + \beta _2}\\
& R_{\beta _1} \boxtimes 1 & \ni & \kappa_i \boxtimes 1, \tau _i \boxtimes 1 & \mapsto & \kappa _i, \tau_i & \in & R _{\beta_1 + \beta _2}\\
& 1 \boxtimes R _{\beta_2} & \ni & 1 \boxtimes \kappa _i, 1 \boxtimes \tau _i & \mapsto & \kappa _{i+n_1}, \tau _{i+n_1} & \in & R _{\beta_1 + \beta _2}
}.
$$
This defines an exact functor
$$\star : R_{\beta_1} \boxtimes R _{\beta _2} \mathchar`-\mathsf{gmod} \ni M _1 \boxtimes M _2 \mapsto R _{\beta _1 + \beta _2} \otimes _{R_{\beta_1} \boxtimes R _{\beta _2}} ( M _1 \boxtimes M _2 ) \in R _{\beta_1 + \beta_2} \mathchar`-\mathsf{gmod}.$$
It is straight-forward to see that $\star$ restricts to an exact functor in the category of graded projective modules:
$$\star : R_{\beta_1} \boxtimes R _{\beta _2} \mathchar`-\mathsf{proj} \ni M _1 \boxtimes M _2 \mapsto R _{\beta _1 + \beta _2} \otimes _{R_{\beta_1} \boxtimes R _{\beta _2}} ( M _1 \boxtimes M _2 ) \in R _{\beta_1 + \beta_2} \mathchar`-\mathsf{proj}.$$
It is straight-forward to define an analogous functor
$$\star : \bigotimes_{i = 1}^{n} R_{\beta_i} \mathchar`-\mathsf{gmod} \rightarrow R _{\beta} \mathchar`-\mathsf{gmod}$$
whenever $\beta = \sum_{i = 1}^n \beta_i$.

If $i \in I$ is a source of $\Gamma$ and $f = ( f_h ) _{h \in \Omega} \in E ^{\Omega}_V$, then we define
$$\epsilon^*_i ( f ) := \dim \ker \bigoplus_{h \in \Omega, h' = i} f _h \le \dim V _i.$$
If $i \in I$ is a sink of $\Gamma$ and $f = ( f_h ) _{h \in \Omega} \in E ^{\Omega}_V$, then we define
$$\epsilon_i ( f ) := \dim \mathrm{coker} \bigoplus_{h \in \Omega, h'' = i} f _h \le \dim V _i.$$
Each of $\epsilon^*_i ( f )$ or $\epsilon _i ( f )$ do not depend on the choice of a point in a $G_V$-orbit. Hence, $\epsilon_i$ or $\epsilon^*_i$ induces a function on $E ^{\Omega}_V$ that is constant on each $G_V$-orbit, and a function on the set of isomorphism classes of simple $G_V$-equivariant perverse sheaves on $E ^{\Omega}_V$ through a unique open dense $G_V$-orbit of its support whenever $i$ is a source or a sink.

\begin{proposition}[Lusztig \cite{Lu3} 6.6]\label{KScrys}
For each $i \in I$, the functions $\epsilon_i$ and $\epsilon_i ^*$ descend to functions on the set of isomorphism classes of simple graded $R_{\beta}$-modules $($up to degree shift$)$.
\end{proposition}

\begin{proof}
Note that \cite{Lu3} 6.6 considers only $\epsilon_i$, but $\epsilon_i ^*$ is obtained by swapping the order of the convolution operation.
\end{proof}

\begin{theorem}[Khovanov-Lauda \cite{KL}, Rouquier \cite{R}, Varagnolo-Vasserot \cite{VV}]\label{VVR} In the above setting, we have:
\begin{enumerate}
\item For each $i \in I$ and $n \ge 0$, $R_{n \alpha _i}$ has a unique indecomposable projective module $P _{n i}$ up to grading shifts;
\item The functor $\star$ induces a $\mathbb Z [ t^{\pm 1} ]$-algebra structure on
$$\mathbf K := \bigoplus _{\beta \in Q^+} K ( R _{\beta} \mathchar`-\mathsf{proj} );$$
\item The algebra $\mathbf K$ is isomorphic to the integral form $U ^+$ of the positive part of the quantized enveloping algebra of type $\Gamma_0$ by identifying $[P_{ni}]$ with the $n$-th divided power of a Chevalley generator of $U ^+$;
\item The above isomorphism identifies the classes of indecomposable graded projective $R_{\beta}$-modules $(\beta \in Q ^+)$ with an element of the lower global basis of $U ^+$ in the sense of $\cite{Ka}$;
\item There exists a set $B ( \infty ) = \bigsqcup _{\beta \in Q ^+} B ( \infty )_{\beta}$ that parameterizes indecomposable projective modules of $\bigoplus _{\beta \in Q ^+} R _{\beta} \mathchar`-\mathsf{gmod}$. This identifies the functions $\epsilon _i, \epsilon ^*_i$ $(i \in I)$ with the corresponding functions on $B ( \infty )$.
\end{enumerate}
\end{theorem}

\begin{proof}
The first assertion is \cite{KL} 2.2 3), the second and the third assertions are \cite{KL} 3.4, and the fourth assertion is \cite{VV} 4.4. Based on this, the fifth follows from Proposition \ref{KScrys}. See also Theorem \ref{crys} in the below.
\end{proof}

\begin{remark}
The coincidence of the lower global basis and the canonical basis is proved by Lusztig \cite{Lu4} and Grojnowski-Lusztig \cite{GL}. We freely utilize this identification in the below.
\end{remark}

\begin{proposition}\label{quiverSC}
In the above setting, the conditions $(\spadesuit)$ and $(\clubsuit)$ are satisfied.
\end{proposition}

\begin{proof}
The condition $(\spadesuit) _1$ is the Gabriel theorem (on the classification of finite algebras, applied to $\mathbb C [ \Gamma ]$). The condition $(\spadesuit) _2$ follows by the fact that $\mathsf{Stab} _G (x_{\lambda})$ is the automorphism group of a $\mathbb C [\Gamma]$-module $\mathtt M$, which must be an open dense part of a linear subspace.

We set $Z ^{\Omega} _{\beta} := F ^{\Omega} _{\beta} \times_{E_{V} ^{\Omega}} F ^{\Omega} _{\beta}$. By \cite{VV} 1.8 (b) and 2.23 (or \cite{CG} 8.6.7), we have an isomorphism $H _{\bullet} ^{G_V} ( Z ^{\Omega} _{\beta} ) \cong \mathrm{Ext} _{G_V} ^{\bullet} ( \mathcal L ^{\Omega} _{\beta}, \mathcal L ^{\Omega} _{\beta} )$ as graded algebras (here we warn that the grading on the LHS is imported from the RHS, and is {\it not} the standard one; cf. \cite{VV} 1.9). Since $G_V$ is a reductive group, we know that each $G_V$-orbit of $( \mathcal B _{\beta} ^{\Omega} )^2$ is an affine bundle over a connected component of $\mathcal B _{\beta} ^{\Omega}$ (see e.g. \cite{CG} \S 3.4). By \cite{VV} 2.11, each fiber of the $G_V$-equivariant map $Z ^{\Omega} _{\beta} \to (\mathcal B _{\beta} ^{\Omega}) ^2$ induced from $\varpi^{\Omega}_{\beta}$ is a vector space. Therefore, we conclude that $Z ^{\Omega} _{\beta}$ is a union of finite increasing sequence of closed subvarieties
$$\emptyset = Z ^{\Omega} _{\beta,0} \subsetneq  Z ^{\Omega} _{\beta,1} \subsetneq  Z ^{\Omega} _{\beta,2} \subsetneq \cdots \subsetneq Z ^{\Omega} _{\beta,\ell}= Z ^{\Omega} _{\beta},$$
where each $Z ^{\Omega} _{\beta,j} \backslash Z ^{\Omega} _{\beta,j-1}$ is an affine bundle over a connected component of $\mathcal B _{\beta} ^{\Omega}$. This implies the purity of $H _{\bullet} ^{G_{V}} ( Z ^{\Omega} _{\beta} )$, and hence $(\clubsuit)_1$ follows.

The condition $(\clubsuit)_2$ is Lusztig \cite{Lu} 10.6.
\end{proof}

\begin{corollary}[Lusztig \cite{Lu}]\label{allp}
Every simple $G_{V}$-equivariant perverse sheaf on $E ^{\Omega} _V$ appears as a non-zero direct summand of $\mathcal L ^{\Omega} _{\beta}$ up to a degree shift.
\end{corollary}

\begin{proof}
By Proposition \ref{quiverSC} and Theorem \ref{VV}, we deduce that the assertion is equivalent to $\# \mathsf{Irr} _0 R_{\beta} = \# ( G_V \backslash E ^{\Omega} _V )$. This follows from a standard bijection between the set of isomorphism classes of indecomposable $\mathbb C [\Gamma]$-modules and a basis of $U ^+$ \`a la Ringel \cite{Ri} (or a consequence of the Gabriel theorem).
\end{proof}

\begin{theorem}[Kashiwara's problem]\label{Kashiwara}
The algebra $R_{\beta}$ has finite global dimension.
\end{theorem}

\begin{proof}
Apply Theorem \ref{absKas} to (\ref{L-auto}), Proposition \ref{quiverSC}, and Corollary \ref{allp}.
\end{proof}

Thanks to Corollary \ref{allp} and Theorem \ref{VVR} 5), we have an identification $B ( \infty ) _{\beta} \cong G _V \backslash E ^{\Omega} _V$, where $V$ is an $I$-graded vector space with $\underline{\dim} \, V = \beta$. By regarding $G_V \backslash E ^{\Omega} _{V}$ as the space of $\mathbb C [\Gamma]$-modules with its dimension vector $\beta$, each $b \in B ( \infty ) _{\beta}$ gives rise to (an isomorphism class of) a $\mathbb C [\Gamma]$-module $\mathtt M _b$. Let us denote by $\mathbb O ^{\Omega} _b$ the $G_V$-orbit of $E_V ^{\Omega}$ corresponding to $b \in B ( \infty ) _{\beta}$. Each $b \in B ( \infty ) _{\beta}$ defines an indecomposable graded projective module $P_b$ of $R_{\beta}$ with simple head $L_{b}$ that is isomorphic to its graded dual $L_b ^*$ (see \S 1).

The standard module $K_b$ and the dual standard module $\widetilde{K}_b$ in \S 1 depends on the choice of $\Omega$ since the Fourier transform interchanges the closure relations. Therefore, we denote by $K ^{\Omega}_b$ (resp. $\widetilde{K} ^{\Omega}_b$) the standard module (resp. the dual standard module) of $L_b$ arising from $E^{\Omega}_V$.

\begin{example}\label{one-root}
If $\beta = m \alpha _i$ for $m \ge 1$ and $i \in I$, then the set $B ( \infty ) _{m \alpha _i}$ is a singleton. Let $L_{mi}$ and $P_{mi}$ be unique simple and projective graded modules of $R_{m \alpha _i}$ up to grading shifts, respectively. The standard module $K_{mi}$ and the dual standard module $\widetilde{K}_{mi}$ do not depend on the choice of $\Omega$ in this case. We have $L_{mi} \cong K_{mi}$ and $P _{mi} \cong \widetilde{K} _{mi}$, and
$$[\widetilde{K}_{mi} : K _{mi}] = \mathsf{gdim} \, \mathbb C [x_1,\cdots,x_m] ^{\mathfrak S_m}.$$
\end{example}

Let $\mathcal Q _{\beta} ^{\Omega}$ be the fullsubcategory of $D _{G_V} ^b ( E ^{\Omega} _{V} )$ consisting all complexes whose direct summands are degree shifts of that of $\mathcal L ^{\Omega} _{\beta}$.

Let $\beta \in Q^+$ with $\mathsf{ht} \, \beta = n$. Let $\le _B$ be the Bruhat order of $\mathfrak S_{n}$ with respect to the set of simple reflections $\{ \sigma_1, \sigma_2, \ldots, \sigma _{n-1} \}$. For each $w \in \mathfrak S_{n}$ and its reduced expression
$$w = \sigma _{j_1} \sigma _{j_2} \cdots \sigma _{j_{L}},$$
we set $\tau_w := \tau _{j_1} \tau_{j_2} \cdots \tau _{j_L}$. Note that $\tau_w$ {\it depends} on the choice of a reduced expression.

\begin{theorem}[Poincar\'e-Birkoff-Witt theorem \cite{KL} 2.7]\label{PBW}
We have equalities as vector spaces:
$$R_{\beta} = \bigoplus _{w \in \mathfrak S_n, \mathbf m \in Y ^{\beta}} \tau_w S _{\beta} e ( \mathbf m ) = \bigoplus _{w \in \mathfrak S_n, \mathbf m \in Y ^{\beta}} S _{\beta} \tau_w e ( \mathbf m ),$$
regardless the choices of $\tau_w$. \hfill $\Box$
\end{theorem}

Let $\beta \in Q ^+$ so that $\mathsf{ht} \, \beta = n$. For each $i \in I$ and $k \ge 0$, we set
\begin{align*}
Y ^{\beta} _{k,i} & := \{ \mathbf m = (m_j) \in Y ^{\beta} \mid m _1 = \cdots = m _k = i \} \text{ and}\\
Y ^{\beta,*} _{k,i} & := \{ \mathbf m = (m_j) \in Y ^{\beta} \mid m _{n} = \cdots = m _{n-k+1} = i \}.
\end{align*}
In addition, we define two idempotents of $R_{\beta}$ as:
$$e_i ( k ) := \sum _{\mathbf m \in Y ^{\beta}_{k,i}} e ( \mathbf m ), \hskip 2mm \text{ and } \hskip 2mm e_i^* ( k ) := \sum _{\mathbf m \in Y ^{\beta,*}_{k,i}} e ( \mathbf m ).$$

\begin{theorem}[Lusztig \cite{Lu3} \S 6, Lauda-Vazirani \cite{LV} 2.5.1]\label{crys}
Let $\beta \in Q_+$ and $i \in I$. For each $b \in B ( \infty ) _{\beta}$ and $i \in I$, we have
\begin{align*}
\epsilon_i ( b ) & = \max \{ k \!\mid e_i ( k ) L _b \neq \{ 0 \} \} \text{ and }\\
\epsilon_i^* ( b ) & = \max \{ k \!\mid e_i^* ( k ) L _b \neq \{ 0 \} \}.
\end{align*}
Moreover, $e_i ( \epsilon_i ( b ) ) L_b$ and $e^*_i ( \epsilon_i^* ( b ) ) L_b$ are irreducible $R_{\epsilon_i ( b ) \alpha _i} \boxtimes R _{\beta - \epsilon_i ( b ) \alpha _i}$-module and $R _{\beta - \epsilon_i^* ( b ) \alpha \i} \boxtimes R _{\epsilon_i^* ( b ) \alpha _i}$-module, respectively. In addition, if we have distinct $b,b' \in B ( \infty )_{\beta}$ so that $\epsilon _i ( b ) = k = \epsilon _i ( b' )$ with $k \ge 0$, then $e_{i} ( k ) L _b$ and $e_i ( k ) L_{b'}$ are not isomorphic as an $R_{k \alpha _i} \boxtimes R _{\beta - k \alpha _i}$-module. \hfill $\Box$
\end{theorem}

\begin{lemma}\label{wt-e}
Let $i \in I$ and let $b_1, b_2 \in B ( \infty )$. Let $L_b$ denote an irreducible constituent of $L_{b_1} \star L_{b_2}$. In case $\epsilon_i ( b_1 ) > 0 = \epsilon_i (b_2)$, then we have $\epsilon_i ( b_1 ) \ge \epsilon_i ( b )$. In case $\epsilon_i ( b_1 ) = 0 = \epsilon_i ( b_2 )$, then we have $\epsilon_i ( b ) = 0$. The same is true if we replace $\epsilon_i$ with $\epsilon^*_i$ and $b_1$ with $b_2$.
\end{lemma}

\begin{proof}
By \cite{KL} \S 2.6, we deduce that the $e ( \mathbf m) ( L_{b_1} \star L_{b_2} ) \neq \{ 0 \}$ implies that $\mathbf m$ is obtained by the shuffle of $\mathbf m_1$ and $\mathbf m_2$ so that $e ( \mathbf m_1 ) L_{b_1} \neq \{ 0 \} \neq e ( \mathbf m_2 ) L_{b_1}$. This yield all the assertions by their definitions.
\end{proof}

\section{Saito reflection functors}
Keep the setting of the previous section. Let $\Omega _i$ be the set of edges $h \in \Omega$ with $h'' = i$ or $h' = i$. Let $s_i \Omega _i$ be a collection of edges obtained from $h \in \Omega _i$ by setting $(s_i h)' = h''$ and $(s_i h)'' = h'$. We define $s_i \Omega := ( \Omega \backslash \Omega _i ) \cup s _i \Omega _i$ and set $s_i \Gamma := ( I, s_i \Omega )$. Note that $\Gamma_0 = ( s_i \Gamma ) _0$.

Let $w_0 \in W$ be the longest element. Choose a reduced expression
$$w_0 = s_{i_1} s_{i_2} \cdots s _{i _{\ell}}.$$
We denote by $\mathbf i := ( i_1, \ldots, i_{\ell} ) \in I ^{\ell}$ the data recording this reduced expression. We say $\mathbf i$ is adapted to $\Omega$ (or $\Gamma$) if each $i_k$ is a sink of $s_{i_{k-1}} \cdots s _{i_1} \Gamma$.

Let $V$ be an $I$-graded vector space with $\underline{\dim} \, V = \beta$. For a sink $i$ of $\Gamma$, we define
$${}_i E _V ^{\Omega} := \bigr\{ ( f _h ) _{h \in \Omega} \in E _V ^{\Omega} \mid \mathrm{coker} ( \bigoplus _{h \in \Omega, h'' = i} f _h : \bigoplus _{h'} V _{h'} \to V _i ) =  \{ 0 \} \bigl\}.$$
For a source $i$ of $\Gamma$, we define
$${}^i E _V ^{\Omega} := \bigr\{ ( f _h ) _{h \in \Omega} \in E _V ^{\Omega} \mid \mathrm{ker} ( \bigoplus _{h \in \Omega, h' = i} f _h : V _i \to \bigoplus _{h''} V _{h''} ) =  \{ 0 \} \bigl\}.$$

Let $\Omega$ be an orientation of $\Gamma$ so that $i \in I$ is a sink. Let $\beta \in Q ^+ \cap s _i Q ^+$. Let $V$ and $V'$ be $I$-graded vector spaces with $\underline{\dim} \, V = \beta$ and $\underline{\dim} \, V' = s_i \beta$, respectively. We fix an isomorphism $\phi : \oplus _{j \neq i} V_j \stackrel{\cong}{\longrightarrow} \oplus _{j \neq i} V'_j$ as $I$-graded vector spaces. We define:
$${}_i Z _{V,V'}^{\Omega} := \Biggl\{ \{ ( f _h ) _{h \in \Omega}, ( f' _h ) _{h \in s_i \Omega}, \psi \} \Biggl| {\small \begin{matrix} ( f _h ) \in {}_i E _V^{\Omega}, ( f'_h ) \in {}^i E _{V'} ^{s_i \Omega}, \\
\phi f_h = f'_{h} \phi \text{ for } h \not\in \Omega_i \\ \psi : V_i' \stackrel{\cong}{\longrightarrow} \mathrm{ker} ( \bigoplus _{h \in \Omega_i} f_h : \bigoplus _h V_{h'} \to V_i )\end{matrix}} \Biggr\}.$$

We have a diagram:
\begin{equation}
\xymatrix{E ^{\Omega} _V & {}_i E ^{\Omega} _V \ar@{_{(}->}[l] _{j_V}& {}_i Z _{V,V'}^{\Omega} \ar@{->>}[r]^{p^i_{V'}} \ar@{->>}[l] _{q^i _{V}} & {}^i E ^{s_i \Omega} _{V'} \ar@{^{(}->}[r] ^{\jmath _{V'}}& E ^{s_i \Omega} _{V'} \hskip 5mm.}\label{df}
\end{equation}
If we set
$$G_{V,V'} := \mathop{GL} ( V_i ) \times \mathop{GL} ( V'_i ) \times \prod _{j\neq i} \mathop{GL} ( V_j ) \cong \mathop{GL} ( V_i ) \times \mathop{GL} ( V'_i ) \times \prod _{j\neq i} \mathop{GL} ( V'_j ),$$
then the maps $p^i_{V'}$ and $q^i_{V}$ are $G_{V,V'}$-equivariant.

\begin{proposition}[Lusztig \cite{Lu2}]\label{bd}
The morphisms $p ^i_V$ and $q ^i _{V}$ in $(\ref{df})$ are $\mathrm{Aut} ( V _i )$-torsor and $\mathrm{Aut} ( V' _i )$-torsor, respectively. \hfill $\Box$
\end{proposition}

When $\beta = \underline{\dim} \, V$, we set
$${} _i R _{\beta} ^{\Omega} := \mathrm{Ext} ^{\bullet} _{G_V} ( j_V ^* \mathcal L ^{\Omega} _{V}, j_V ^* \mathcal L ^{\Omega} _{V} ) \hskip 2mm \text{ and } \hskip 2mm {} ^i R _{s_i \beta} ^{s_i \Omega} := \mathrm{Ext} ^{\bullet} _{G_{V'}} ( \jmath _{V'}^* \mathcal L ^{s_i \Omega} _{V'}, \jmath _{V'} ^* \mathcal L ^{s_i \Omega} _{V'} )$$
for the time being (see Corollary \ref{indep}).

\begin{lemma}\label{i-quotients}
We have an algebra isomorphism ${} _i R _{\beta} ^{\Omega} \cong R_{\beta} / ( R_{\beta} e_i ( 1 ) R_{\beta} )$. Similarly, the algebra ${} ^i R _{s_i \beta} ^{s_i \Omega}$ is isomorphic to $R_{s_i \beta} / ( R_{s_i \beta} e_i ^* ( 1 ) R_{s_i \beta} )$. 
\end{lemma}

\begin{proof}
The maps $j_V$ and $\jmath _{V'}$ are $G_V$- and $G_{V'}$-equivariant open embeddings, respectively. Therefore, we apply Lemma \ref{localKsc} and Proposition \ref{Kquotients} repeatedly to deduce ${} _i R _{\beta} ^{\Omega} \cong R_{\beta} / ( R_{\beta} e R_{\beta} )$, where $e \in R_{\beta}$ is a degree zero idempotent so that $e L_{b} = L_{b}$ ($\mathbb O_b ^{\Omega} \not\subset \mathrm{Im} \, j_V$) or $\{ 0 \}$ ($\mathbb O_b ^{\Omega} \subset \mathrm{Im} \, j_V$). By Proposition \ref{KScrys}, Theorem \ref{VVR} 5), and Theorem \ref{crys}, we conclude $R _{\beta} e R_{\beta} = R _{\beta} e _i ( 1 ) R _{\beta}$, which proves the first assertion. The case of ${} ^i R _{s_i \beta} ^{s_i \Omega}$ is similar, and we omit the detail.
\end{proof}

\begin{corollary}
The set of isomorphism classes of graded simple modules of ${} _i R _{\beta} ^{\Omega}$ and ${} ^i R _{s_i \beta} ^{s_i \Omega}$ are $\{ L_b \left< j \right>\} _{\epsilon _i ( b ) = 0, j \in \mathbb Z}$ and $\{ L_b \left< j \right>\} _{\epsilon ^* _i ( b ) = 0, j \in \mathbb Z}$, respectively. \hfill $\Box$
\end{corollary}

\begin{corollary}\label{indep}
The algebras ${} _i R _{\beta} ^{\Omega}$ and ${} ^i R _{s_i \beta} ^{s_i \Omega}$ do not depend on the choice of $\Omega$. \hfill $\Box$
\end{corollary}

\begin{proposition}\label{TM}
In the setting of Proposition \ref{bd}, two graded algebras ${} _i R _{\beta} ^{\Omega}$ and ${} ^i R _{s_i \beta} ^{s_i \Omega}$ are Morita equivalent to each other. In addition, this Morita equivalence is independent of the choice of $\Omega$ $($as long as $i$ is a sink$)$.
\end{proposition}

\begin{proof}
First, note that the maps $j_V, \jmath_{V'}$ are open embeddings. In particular, $j_V ^* \mathcal L ^{\Omega} _{V}$ and $\jmath _{V'}^* \mathcal L ^{s_i \Omega} _{V'}$ are again direct sums of shifted equivariant perverse sheaves. By Proposition \ref{bd} and \cite{BL} 2.2.5, we have equivalences
$$D _{G_V}^b ( {}_i E ^{\Omega} _V ) \stackrel{( q^i_V )^*}{\longrightarrow} D _{G_{V,V'}}^b ( {}_i Z ^{\Omega} _{V,V'} ) \stackrel{( p^i_{V'} )^*}{\longleftarrow} D _{G_{V'}}^b ( {}^i E ^{s_i \Omega} _{V'} ).$$
In addition, a simple $G_{V,V'}$-equivariant perverse sheaf $\mathcal L$ on ${}_i Z_{V,V'} ^{\Omega}$ admits isomorphisms
\begin{equation}
( q^i_V )^* \left( {}_i \mathcal L \, [\dim \mathop{GL} (V'_i)] \right) \cong \mathcal L \cong ( p^i_{V'} )^*  \left( {}^i \mathcal L \, [\dim \mathop{GL} (V_i)] \right),\label{Scorr}
\end{equation}
where ${}_i \mathcal L$ and ${}^i \mathcal L$ are simple $G_V$- and $G_{V'}$-equivariant perverse sheaves on ${}_i E ^{\Omega} _V$ and ${}^i E ^{s_i \Omega} _{V'}$, respectively. These induce isomorphisms of algebras:
\begin{equation}
B_{( G_V, {}_i E ^{\Omega} _V )} \cong B _{(G_{V,V'}, {}_i Z ^{\Omega} _{V,V'})} \cong B _{(G_{V'}, {}^i E ^{s_i \Omega} _{V'} )}.\label{Acorr}
\end{equation}
Therefore, $B_{( G_V, {}_i E ^{\Omega} _V )}$ and $B _{(G_{V'}, {}^i E ^{s_i \Omega} _{V'} )}$ are Morita equivalent to the algebras in the assertion by Corollary \ref{allp}, which implies the first assertion.

We prove the second assertion. For any two orientations $\Omega$ and $\Omega'$ which have $i$ as a common sink, we have Fourier transforms $\mathcal F ^{\Omega}$ and $\mathcal F ^{s_i \Omega}$ so that $\mathcal F ^{\Omega} ( \mathcal L ^{\Omega} _{\beta} ) = \mathcal L ^{\Omega'} _{\beta}$ and $\mathcal F ^{s_i \Omega} ( \mathcal L ^{s_i \Omega} _{s_i \beta} ) = \mathcal L ^{s_i \Omega'} _{s_i \beta}$. Since $\Omega _i = \Omega _i'$, these two Fourier transforms are induced by the pairing between direct summands $E \subset E_V^{\Omega}$ and $E^* \subset E_V ^{\Omega'}$ which can be identified with those of $E_{V'}^{s_i \Omega}$ and $E_{V'} ^{s_i \Omega'}$ in (\ref{df}) via $\phi$. Since the diagram (\ref{df}) is the product of a vector space and the contribution from $\Omega_i$, we conclude that two pairs of sheaves $( \mathcal L ^{\Omega} _{\beta}, \mathcal L ^{s_i \Omega} _{s_i \beta})$ and $( \mathcal L ^{\Omega'} _{\beta}, \mathcal L ^{s_i \Omega'} _{s_i \beta})$ are exchanged by $\mathcal F ^{\Omega}$ and $\mathcal F ^{s_i \Omega}$ commuting with the diagram (\ref{df}). This identifies the Morita equivalences obtained by $\Omega$ and $\Omega'$ as required.
\end{proof}


The maps $q^i_V$ and $p^i_{V'}$ give rise to a correspondence between orbits. For each $b \in B ( \infty )_{s_i \beta}$, we denote by $T_i ( b ) \in B ( \infty ) _{\beta} \sqcup \{ \emptyset \}$ the element so that $( p ^i _{V'} )^{-1} ( \mathbb O _b ^{s _i \Omega} ) \cong ( q^i _V )^{-1} ( \mathbb O _{T _i ( b )} ^{\Omega})$ (we understand that $T _i ( b ) = \emptyset$ if $\mathbb O ^{s _i \Omega} _{b} \not\subset \mathrm{Im} \, p ^i _{V'}$). Note that $T_i ( b ) = \emptyset$ if and only if $\epsilon ^* _i ( b ) > 0$. In addition, we have $\epsilon _i ( T_i ( b ) ) = 0$ if $T_i ( b ) \neq \emptyset$. We set $T_i ^{-1} (b') := b$ if $b' = T_i ( b ) \neq \emptyset$.

Thanks to Corollary \ref{indep}, we can drop $\Omega$ or $s_i \Omega$ from ${} _i R _{\beta} ^{\Omega}$ and ${} ^i R _{\beta} ^{s _i \Omega}$. We define a left exact functor
$$\mathbb T^* _i : R _{\beta} \mathchar`-\mathsf{gmod} \longrightarrow \!\!\!\!\! \rightarrow {} _i R _{\beta}\mathchar`-\mathsf{gmod} \stackrel{\cong}{\longrightarrow} {}^i R_{s_i \beta} \mathchar`-\mathsf{gmod} \hookrightarrow R_{s_{i} \beta} \mathchar`-\mathsf{gmod},$$
where the first functor is $\mathrm{Hom} _{R_{\beta}} ( {} _i R_{\beta}, \bullet )$, the second functor is Proposition \ref{TM}, and the third functor is the pullback. Similarly, we define a right exact functor
$$\mathbb T _i : R _{\beta} \mathchar`-\mathsf{gmod} \longrightarrow \!\!\!\!\! \rightarrow {} ^i R _{\beta}\mathchar`-\mathsf{gmod} \stackrel{\cong}{\longrightarrow} {}_i R_{s_i \beta} \mathchar`-\mathsf{gmod} \hookrightarrow R_{s_{i} \beta} \mathchar`-\mathsf{gmod},$$
where the first functor is ${}^i R_{\beta} \otimes_{R_{\beta}} \bullet$. We call these functors the Saito reflection functors (cf. \cite{S}). By the latter part of Proposition \ref{TM}, we see that these functors are independent of the choices involved.

Let $i \in I$. We define $R _{\beta} \mathchar`-\mathsf{gmod} _i$ (resp. $R _{\beta} \mathchar`-\mathsf{gmod} ^{i}$) to be the fullsubcategory of $R _{\beta} \mathchar`-\mathsf{gmod}$ so that each simple subquotient is of the form $L_b \left< k \right>$ ($k \in \mathbb Z$) with $b \in B ( \infty )_{\beta}$ that satisfies $\epsilon _i ( b ) = 0$ (resp. $\epsilon _i^* ( b ) = 0$). In addition, for each $i \neq j \in I$, we define $R _{\beta} \mathchar`-\mathsf{gmod} _{j} ^i := R _{\beta} \mathchar`-\mathsf{gmod} ^{i} \cap R _{\beta} \mathchar`-\mathsf{gmod} _{j}$.

\begin{theorem}[Saito reflection functors]\label{SRF}
Let $i \in I$. We have:
\begin{enumerate}
\item Assume that $i$ is a source of $\Omega$. For each $b \in B ( \infty )_{\beta}$, we have
$$\mathbb T_i K _b ^{\Omega} = \begin{cases} K _{T_i(b)} ^{s_i \Omega} & (\epsilon^*_i (b) = 0)\\ \{0\} & (\epsilon^*_i (b) > 0)\end{cases};$$
\item For each $b \in B ( \infty )_{\beta}$, we have
$$\mathbb T_i L _b = \begin{cases} L_{T_i(b)} & (\epsilon^*_i (b) = 0)\\ \{0\} & (\epsilon^*_i (b) > 0)\end{cases}, \text{ and } \hskip 3mm \mathbb T_i^* L _b = \begin{cases} L_{T_i^{-1}(b)} & (\epsilon_i (b) = 0)\\ \{0\} & (\epsilon_i (b) > 0)\end{cases};$$
\item The functors $( \mathbb T_i, \mathbb T^*_i)$ form an adjoint pair;
\item For each $M \in R _{\beta} \mathchar`-\mathsf{gmod} ^i$ and $N \in R_{s_i \beta} \mathchar`-\mathsf{gmod} _i$, we have
$$\mathrm{ext} _{R_{s_{i} \beta}} ^{*} ( \mathbb T _i M, N ) \cong \mathrm{ext} _{R_{\beta}} ^{*} ( M, \mathbb T^* _i N );$$
\item Let $i \neq j \in I$. For each $\beta \in Q^+$ and $m \ge 0$, we have
$$\mathbb T_i ( P_{mj} \star M ) \cong ( \mathbb T_i P_{mj} ) \star \mathbb T_i M$$
as graded $R_{s_i ( \beta + m \alpha _j )}$-modules for every $M \in R _{\beta} \mathchar`-\mathsf{gmod} _j$.
\end{enumerate}
\end{theorem}

\begin{remark}
The proof of Theorem \ref{SRF} is given by two parts, namely 1)--4) and 5). We warn that the proof of the latter part rests on the earlier part.
\end{remark}

\begin{proof}[Proof of Theorem \ref{SRF} 1)--4)]
We prove the first assertion. The subset ${}^i E _V ^{\Omega} \subset E _V ^{\Omega}$ (with $\underline{\dim} \, V = \beta$) is open. Therefore, Theorem \ref{KS} 1) asserts that $[K _b ^{\Omega} : L _{b'}] \neq 0$ only if $\epsilon _i ^* ( b' ) = 0$ whenever $\epsilon _i ^* ( b ) = 0$. It follows that ${}^i R_{\beta} \otimes_{R_{\beta}} K _b ^{\Omega} \cong K _b ^{\Omega}$ as a vector space if $\epsilon ^* _i ( b ) = 0$, and $\{ 0 \}$ otherwise. This gives rise to a standard module of ${}^i R _{\beta}$ by Lemma \ref{localKsc}, and thus it gives a standard module of ${}_i R_{s_i \beta}$ by Proposition \ref{TM}. Note that the subset ${}_i E _{V'} ^{s_i \Omega} \subset E _{V'} ^{s_i \Omega}$ (with $\underline{\dim} \, V' = s_i \beta$) is also open. Therefore, we use Lemma \ref{i-quotients} to deduce the first assertion.

The second assertion is immediate from the first assertion and the construction of $\mathbb T_i$ and $\mathbb T_i^*$.

We prove the third assertion. By Lemma \ref{i-quotients}, we know that $\mathbb T_i$ factors through the functor giving the maximal quotient which is a ${} ^i R _{\beta}$-module, while $\mathbb T_i^*$ factors through the functor giving the maximal submodule which is an ${}_i R_{\beta}$-module. Therefore, the third assertion follows by the Morita equivalence ${} ^i R_{\beta} \mathchar`-\mathsf{gmod} \cong {}_i R_{s_i \beta} \mathchar`-\mathsf{gmod}$ for every $\beta \in Q^+ \cap s_i Q^+$.

For the fourth assertion, notice that $R_{\beta}$- and $R_{s_i \beta}$-action on $M$ and $N$ factors through ${}^i R_{\beta}$ and ${}_i R_{s_i \beta}$, respectively. It follows that ${}^i R_{\beta} \otimes _{R_{\beta}} M \cong M$, ${}_i R_{s_i \beta} \otimes _{R_{s_i \beta}} \mathbb T_i M \cong \mathbb T_i M$, and $\mathrm{Hom} _{R_{s_i \beta}} ( {} _i R_{s_i \beta}, N ) \cong N$. By Lemma \ref{i-quotients} and Corollary \ref{isomKext}, we deduce that each indecomposable projective ${}_i R_{s_i \beta}$-module ${} _i P$ admits an $R_{s_i \beta}$-graded projective resolution
$$\cdots \to P_2 \to P_1 \to P_0 \to {}_i P \to 0$$
so that $P_0$ is indecomposable and ${}_i R_{s_i\beta} \otimes_{R_{s_i\beta}} P_k = \{ 0 \}$ for $k \ge 1$. Therefore, we have
$$\mathrm{ext} _{R_{s_i \beta}} ^{*} ( M, N ) \cong \mathrm{ext} _{{}_i R_{s_i \beta}} ^{*} ( M, N ),$$
where we regard $M, N$ as ${}_i R_{s_i\beta}$-modules via Proposition \ref{TM} (here we treat the Morita equivalence as an isomorphism for simplicity). Applying the same argument for ${}^i R_{\beta}$ (again for $M$), we conclude the result.
\end{proof}

\begin{lemma}\label{indpro}
Let $i \in I$. For each $\beta \in Q^+$, $m \ge 0$, and an indecomposable graded projective ${}_i R_{\beta}$-module $P$, the module $P _{mi} \star P$ is an $R_{\beta + m \alpha_i}$-module with simple head.
\end{lemma}

\begin{proof}
By the Frobenius reciprocity, we have
\begin{equation}
\mathrm{hom} _{R_{\beta + m \alpha _i}} ( P _{m i} \star P, L _{b} ) \cong \mathrm{hom} _{R_{m \alpha _i} \boxtimes R_{\beta}} ( P _{m i} \boxtimes P, L _{b} )\label{FRind}
\end{equation}
for every $b \in B ( \infty ) _{\beta + m \alpha _i}$. Assume that the above space is non-zero to deduce the uniqueness of $b$ and the one-dimensionality of $(\ref{FRind})$. Choose $d \in B ( \infty )_{\beta}$ so that $L_{d}$ is the unique simple quotient of $P$. We have $\epsilon _i ( d ) = 0$ by assumption. By Theorem \ref{PBW}, we have $e ( \mathbf m) (P _{m i} \star P ) \neq \{ 0 \}$ only if there exist a minimal length representative $w \in \mathfrak S_{( \mathsf{ht} \, \beta + m )} / ( \mathfrak S_m \times \mathfrak S _{\mathsf{ht} \, \beta} )$ and $\mathbf m' \in Y ^{\beta}$ so that $e ( \mathbf m' ) P \neq \{ 0 \}$ and $\mathbf m = w ( mi + \mathbf m' )$. Since $\mathbf m' \not\in Y ^{\beta}_{1,i}$, we deduce $\epsilon _{i} ( b ) \le m$. Thus, if (\ref{FRind}) is non-trivial, then we have $\epsilon _i ( b ) = m$ and $w = 1$. Now Theorem \ref{crys} forces $e _i ( m ) L_b \cong L_{mi} \boxtimes L_{d}$. Therefore, $P_{m i} \star P$ has at most one quotient, which completes the proof.
\end{proof}

\begin{theorem}[Induction theorem]\label{induction}
Let $V (i)$ be $I$-graded vector spaces with $\underline{\dim} \, V (i) = \beta _i$, and $b_i \in B ( \infty )_{\beta_i}$ for $i = 1,2$. Let $b \in B ( \infty )_{\beta_1 + \beta _2}$ so that $\mathtt M _{b} \cong \mathtt M _{b_1} \oplus \mathtt M_{b_2}$ as $\mathbb C [\Gamma]$-modules. Assume the following condition $(\star)$:
\begin{itemize}
\item[$(\star)_0$] $M_{b_1'}$ is not a quotient of $M_b$ for every $b_1 \neq b_1' \in B ( \infty ) _{\beta_1}$, and $M_{b_2'}$ is not a submodule of $M_b$ for every $b_2 \neq b_2' \in B ( \infty ) _{\beta_2}$;
\item[$(\star)_1$] $\mathrm{Ext} ^1 _{\mathbb C [\Gamma]} ( \mathtt M_{b_1}, \mathtt M_{b_2} ) = \{ 0 \}$.
\end{itemize}
We have an isomorphism $K_{b_1} ^{\Omega} \star K_{b_2} ^{\Omega} \cong K _{b} ^{\Omega}$ as an ungraded $R_{\beta_1 + \beta_2}$-module.

In addition, if $\mathtt M_{b}$ canonically determines the factor $\mathtt M _{b_2}$ as a vector subspace, then $(\star)_0$ and $(\star)_1$ implies
$$K_{b_1} ^{\Omega} \star K_{b_2} ^{\Omega} \cong K _{b} ^{\Omega}$$
as a graded $R_{\beta_1 + \beta_2}$-module.
\end{theorem}

Before proving Theorem \ref{induction}, we present some of its consequences. The proof of Theorem \ref{induction} itself is given at the end of this section.

\begin{corollary}\label{i-ind}
Suppose that $i$ is a sink of $\Omega$. Let $m \ge 0$. For each $\beta \in Q^+$ and $b \in B ( \infty ) _{\beta}$ with $\epsilon _i ( b ) = 0$, the module $K_{mi} \star K_{b} ^{\Omega}$ is an indecomposable graded $R_{\beta + m \alpha_i}$-module isomorphic to $K_{b'} ^{\Omega}$ with $\mathtt M _{b'} \cong \mathtt M _i ^{\oplus m} \oplus \mathtt M _b$.
\end{corollary}

\begin{proof}
By Example \ref{one-root}, we deduce that the first part of $(\star)_0$ is a void condition. Every irreducible subquotient of a $\mathbb C [\Gamma]$-module isomorphic to $\mathtt M_i$ is in its socle. Hence, the second part of $(\star)_0$ follows by the comparison of the socles. Since $i$ is a sink, we have no extension of $\mathtt M_{i} ^{\oplus m}$ by $\mathtt M _b$, which is $(\star)_1$. We write $\beta = k \alpha _i + \sum _{j \neq i} k _j \alpha _j$. Since $\epsilon_i ( b ) = 0$, $\mathtt M _{i}$ is not a direct summand of $\mathtt M _{b}$. In particular, $M_{b}$ is canonically determined by $M_{b'}$ as its direct factor. Applying Theorem \ref{induction} yields the result.
\end{proof}

\begin{corollary}\label{j-ind}
Assume that $i$ is a source and $j$ is a sink of $\Omega$. Let $\beta \in Q^+$. For each $m \ge 0$ and $b \in B ( \infty ) _{\beta}$ such that $\epsilon _j ( b ) = 0$, we have
$$\mathbb T_i ( K_{mj} \star K_b ^{\Omega} ) \cong ( \mathbb T_i K_{mj} ) \star \mathbb T_i K_b ^{\Omega}$$
as graded $R_{s_i ( \beta + m \alpha _j)}$-modules.
\end{corollary}

\begin{proof}
By Corollary \ref{i-ind}, we see that $K_{mj} \star K_b ^{\Omega} \cong K _{b'} ^{\Omega}$, where $\mathtt M_{b'} \cong \mathtt M_j ^{\oplus m} \oplus \mathtt M_b$. By \cite{Lu} 4.4 (c), we deduce that $T_i ( b' ) \neq \emptyset$ if and only if $T_i ( b ) \neq \emptyset$. Since a standard module is generated by its simple head, we deduce that $\mathbb T_i ( K_{mj} \star K_b ^{\Omega} ) = \{ 0 \}$ if $\epsilon _i^* ( b ) > 0$, and it is isomorphic to $K_{T_i ( b' )} ^{s_i \Omega}$ if $\epsilon^* _i ( b ) = 0$.

Since $i \neq j$, we always have $\mathbb T_i K _{mj} \neq \{ 0 \}$. Therefore, we conclude that the RHS is non-zero if and only if the LHS is non-zero. Thus, it suffices to show that the RHS is isomorphic to $K _{T_i ( b')} ^{s_i \Omega}$.

If we have $i \not\leftrightarrow j$, then $j$ is a sink of $s_i \Gamma$. By $\epsilon _j ( b ) = 0$ and the assumption, we deduce that $\mathtt M_{T_i ( b )}$ also do not contain $\mathtt M_j$ in this case. Hence, we deduce $\epsilon _j ( T_i ( b ) ) = 0$. In addition, we have $\mathbb T_i K_{mj} ^{\Omega} \cong K _{mj} ^{s_i \Omega}$. Therefore, we apply Corollary \ref{i-ind} to deduce that the RHS is $K _{T_i ( b')} ^{s_i \Omega}$.

Assume that we have $i \leftrightarrow j$. Let $\mathtt M_{i,j}$ be a unique indecomposable $\mathbb C [s_i \Gamma]$-module with $\underline{\dim} \, \mathtt M_{i,j} = \alpha _i + \alpha _j$ (up to an isomorphism). By $\epsilon _j ( b ) = 0$ and {\it loc. cit.} 4.4 (c), we conclude that $\mathtt M_{T_i ( b )}$ does not contain $\mathtt M _i, \mathtt M _{i,j}$ as its direct factor. By assumption, $i$ is a sink of $s_i \Gamma$ and $j$ is a source of an edge from $j$ to $i$, but is a source of no other edges. This particularly implies that $\mathtt M_{i}$ is the socle of $\mathtt M_{i,j}$. Therefore, we conclude the first half of $(\star)_0$ in Theorem \ref{induction}. If an indecomposable $\mathbb C[s_i \Gamma]$-module contains $\mathtt M_i$ or $\mathtt M_{i,j}$ as its subquotient, then it must be a submodule. If an indecomposable $\mathbb C [s_i \Gamma]$-module has a non-zero homomorphism to $\mathtt M_i$ or $\mathtt M_{i,j}$, then it must be isomorphic to either $\mathtt M_i$ or $\mathtt M_{i,j}$. These imply the latter half of $(\star)_0$ in Theorem \ref{induction}. In addition, we have
$$\mathrm{Ext} ^1 _{\mathbb C [s_i \Gamma]} ( \mathtt M_{i,j} ^{\oplus m}, \mathtt M _{T_i ( b )}) = \{ 0 \}.$$
Therefore, we conclude $(\star)_1$ in Theorem \ref{induction}. Let $h_* \in s_i \Omega$ be the unique edge so that $h_*' = j, h _*'' = i$. For a representation $(f_h) _{h \in s _i \Omega}$ on $V = \bigoplus _{i \in I} V_i$ isomorphic to $\mathtt M _{i,j} ^{\oplus m} \oplus \mathtt M _{T_i ( b )}$, we set
$$V' _k := \begin{cases} V _k & (k \neq i, j)\\ \mathrm{Im} \bigl( \bigoplus _{h \in s_i \Omega, h'' = i} f_h \oplus \bigoplus _{h \in s_i \Omega, h'' = j} f_{h_*} f_h \bigr) & (k = i) \\
\mathrm{Im} \bigoplus _{h \in s_i \Omega, h'' = j} f_h + f_{h_*}^{-1} ( V'_i ) + \ker f _{h_*} & (k = j)\end{cases}.$$
Then, the space $V' \subset V$ defines a canonical $\mathbb C [s_i \Gamma]$-submodule $\mathtt M'$ on $V'$ so that $\mathtt M' \cong \mathtt M _{T_i ( b )}$. Therefore, we conclude that $( \mathbb T_i K_{mj} ) \star \mathbb T_i K_b ^{\Omega} \cong K _{T_i ( b' )} ^{s_i \Omega}$ as required.
\end{proof}

\begin{lemma}\label{i-factor}
Let $i,j \in I$ be distinct vertices, $m \ge 0$, and $\beta \in Q ^+$. For each $b \in B ( \infty )$ so that $\epsilon _i ( b ) > 0$, the module $\mathbb T_i K_{m j} \star L_b$ has simple head that is isomorphic to $L_{b'}$ with $\epsilon _i ( b' ) > 0$ up to grading shifts.
\end{lemma}

\begin{proof}
We first consider the case $i \not\leftrightarrow j$. We assume that both $i$ and $j$ are sink. We have $\mathbb T_i K _{m j} \cong K _{m j}$. By Theorem \ref{induction}, we further deduce an isomorphism $K _{m j} \star K _{p j} \cong K _{(m+p) j}$ for $p \ge 0$. Together with Corollary \ref{i-ind} and the induction-by-stage argument, we conclude that $K_{m j} \star L_b$ has simple head $L_{b'}$. Moreover, we have $\mathtt M _{b'} \cong \mathtt M_{j} ^{\oplus m} \oplus \mathtt M_{b}$. Therefore, we have $\epsilon _i ( b ) > 0$ if and only if $\epsilon _i ( b') > 0$ since $\epsilon _i$ counts the number of direct summand isomorphic to $\mathtt M_i$ by our assumption on $\Omega$.

Now we consider the case $i \leftrightarrow j$. We rearrange $\Omega$ so that $j$ is a sink of $\Omega$ and $i$ is sink of $s_i \Omega$, and employ the same notation as in the proof of Corollary \ref{j-ind}. We have a decomposition
$$\mathtt M _b \cong \mathtt M _{i} ^{\oplus p} \oplus \mathtt M_{i,j} ^{\oplus q} \oplus \mathtt M _{d} \hskip 3mm \text{ with } \hskip 3mm d \in B ( \infty ) _{\beta - p \alpha _i - q s _i \alpha _j}$$
as $\mathbb C [s_i \Gamma]$-modules so that $\mathtt M _d$ does not contain $\mathtt M_i$ or $\mathtt M_{i,j}$ as its direct summand. Here we have $d = T_i ( f )$ with $\epsilon_j ( f ) = 0$ by \cite{Lu} 4.4 (c). We set $d' \in B ( \infty ) _{m s_i \alpha _j + \beta}$ so that $\mathtt M _{d'} \cong \mathtt M _{i,j} ^{\oplus m} \oplus \mathtt M _b$. Thanks to Corollary \ref{i-ind} and Corollary \ref{j-ind}, we have
$$K _b ^{s_i \Omega} \cong K _{p i} \star ( \mathbb T_i K _{q j} ) \star K _{d} ^{s_i \Omega}.$$
By Corollary \ref{i-ind}, we deduce that $K _{i} \star \mathbb T_i K _{j}$ is isomorphic to a standard module of $R_{2 \alpha _i + \alpha _j}$. Since the orbit corresponding to $K _{i} \star \mathbb T_i K _{j}$ is open dense, we deduce that $K _{i} \star \mathbb T_i K _{j}$ is simple. By inspection, we find that $\# \mathsf{Irr} _0 R_{2\alpha _i + \alpha _j} = 2$ and each of simple graded $R_{\alpha _i + 2 \alpha _j}$-module has dimension $3$. Hence, $\mathbb T_i K_{j} \star K _{i}$ must be simple. By a weight comparison argument, we deduce that $K _{i} \star \mathbb T_i K _{j} \cong \mathbb T_i K _{j} \star K _{i} \left< 1 \right>$. By Theorem \ref{induction}, we deduce that
$$( \mathbb T_i K_{r j} ) \star ( \mathbb T_i K_{s j} ) \cong \mathbb T_i K _{(r+s)j} \hskip 5mm \text{for every }r,s \ge 0.$$
Hence, we deduce $K _{p i} \star \mathbb T_i K _{mj} \cong \mathbb T_i K _{mj} \star K _{pi}$ up to grading shifts by induction.

Therefore, the induction-by-stage implies that the ungraded $R_{\beta + m s_i \alpha _j}$-module
$$\mathbb T_i K_{m j} \star K_b ^{s_i \Omega} \cong \mathbb T_i K_{m j} \star K _{p i} \star ( \mathbb T_i K _{q j} ) \star K _{d} ^{s_i \Omega} \cong K _{p i} \star ( \mathbb T_i K _{(m+q) j} ) \star K _{d} ^{s_i \Omega} \cong K _{d'} ^{s_i \Omega}$$
has simple head $L_{b'}$ with $\epsilon _i ( b' ) = p > 0$ as desired.
\end{proof}

\begin{lemma}\label{premon}
For each $\beta_1,\beta_2 \in Q^+$, we have a canonical surjection
$$\mathbb T_i ( M_1 \star M_2 ) \longrightarrow \!\!\!\!\! \rightarrow ( \mathbb T_i M_1 ) \star ( \mathbb T_i M_2 )$$
as graded $R_{s_i ( \beta_1 + \beta_2 )}$-modules for every $M_1 \in R _{\beta_1} \mathchar`-\mathsf{gmod}$ and $M_2 \in R _{\beta_2} \mathchar`-\mathsf{gmod}$.
\end{lemma}

\begin{proof}
Put $\beta := \beta_1 + \beta_2$. The induction functor $\star$ is represented by the $(R_{\beta}, R_{\beta_1} \boxtimes R_{\beta_2})$-bimodule $R_{\beta} e_1$, where $e_1$ is an idempotent. The Saito reflection functor $\mathbb T_i$ factors through the quotient by $R_{\beta} e^*_i ( 1 ) R_{\beta}$. Therefore, the two compositions are realized as
$$R_{\beta} e_1 / R_{\beta} e^*_i ( 1 ) R_{\beta} e_1 \hskip 3mm \text{and} \hskip 3mm R_{\beta} e_1 \otimes \left( ( R_{\beta_1} / R_{\beta_1} e^*_i ( 1 ) R_{\beta_1} ) \boxtimes ( R_{\beta_2} / R_{\beta_2} e^*_i ( 1 ) R_{\beta_2} ) \right) e_1,$$
respectively. By Lemma \ref{wt-e}, we know that an irreducible direct summand of the head $L_b$ of the induction of two simple modules $L_{b_1}$ and $L_{b_2}$ satisfies $\epsilon^*_i (b)> 0$ only if $\epsilon^*_i ( b _1 ) > 0$ or $\epsilon^*_i ( b_2 ) > 0$. Therefore, the RHS annihilates more simple modules than these from the LHS, and hence we obtain a surjection as required.
\end{proof}

\begin{proof}[Proof of Theorem \ref{SRF} 5)]
We choose an orientation $\Omega$ so that $i$ is a source and $j$ is a sink. Let $F_1 := ( \mathbb T_i P_{mj} ) \star ( \mathbb T_i \bullet )$ and $F_2 := \mathbb T_i ( P_{mj} \star \bullet )$ be two functors from $R _{\beta} \mathchar`-\mathsf{gmod}_j$ to $R_{s_i (\beta + m \alpha _j )} \mathchar`-\mathsf{gmod}$. Both of them are exact on $R _{\beta} \mathchar`-\mathsf{gmod} _j ^i$. Therefore, taking successive quotients of the isomorphisms in Corollary \ref{j-ind} (cf. Theorem \ref{SRF} 1)) yield
\begin{equation}
\mathbb T_i ( K_{mj} \star L_b ) \cong ( \mathbb T_i K_{mj} ) \star \mathbb T_i L_b\label{psimple}
\end{equation}
as a graded $R_{s_i ( \beta + m \alpha _j)}$-module for every $b \in B ( \infty )_{\beta}$ such that $\epsilon _j ( b ) = 0$. By Lemma \ref{premon}, we have a natural transformation
$$F_2 = \mathbb T_i ( P_{mj} \star \bullet ) \longrightarrow ( \mathbb T_i P_{mj} ) \star \mathbb T_i \bullet = F_1.$$
Thanks to Lemma \ref{indpro}, we see that $F_2$ sends an indecomposable projective module of ${}_j R_{\beta}$ (regarded as an $R_{\beta}$-module) to a module with simple head (or zero). The image of this simple head survives under this natural transformation by $(\ref{psimple})$. This forces two functors $F_1$ and $F_2$ to be isomorphic on projective objects of $R _{\beta} \mathchar`-\mathsf{gmod} _j$ by the comparison of their graded characters. Therefore, we conclude that they are isomorphic.
\end{proof}

The rest of this section is devoted to the proof of Theorem \ref{induction}. During the proof of Theorem \ref{induction}, we omit $\Omega$ from the notation. We set $\beta := \beta _1 + \beta _2$, and $V := V(1) \oplus V(2)$. We set $n = \mathsf{ht} \, \beta$, and $n_i := \mathsf{ht} \, \beta _i$ for $i=1,2$. We write $\beta _i = \sum _{j \in I} d _i ( j ) \alpha _j$ for $i = \emptyset, 1, 2$.

We recall the convolution operation from Lusztig \cite{Lu}.

We consider two varieties with natural $G_V$-actions:
\begin{align*}
\mathrm{Gr}_{V(1),V(2)} ( V ) & := \Bigl\{ ( F, x, \psi_{1}, \psi _{2} ) \Bigl| {\footnotesize \begin{matrix} F \subset V \text{ : $I$-graded vector subspace} \\ x \in E_V, \text{ s.t. } x F \subset F \\\psi _1 : V/F \cong V(1), \psi_2 : F \cong V(2)  \end{matrix}} \Bigr\},\\
\mathrm{Gr} _{\beta_1,\beta_2} ( V ) & := \Bigl\{ ( F, x ) \Bigl| {\footnotesize \begin{matrix} F \subset V \text{ : $I$-graded vector subspace} \\ x \in E_V, \text{ s.t. } x F \subset F \\ \underline{\dim} \, F = \beta _2  \end{matrix}} \Bigr\}.
\end{align*}
We have a $G_{V(1)} \times G_{V(2)}$-torsor structure $\vartheta : \mathrm{Gr}_{V(1),V(2)} ( V ) \longrightarrow \mathrm{Gr} _{\beta_1,\beta_2} ( V )$ given by forgetting $\psi_1$ and $\psi _2$. We have two maps
\begin{align*}
\mathsf{p}  :  \, & \mathrm{Gr} _{\beta_1,\beta_2} ( V ) \ni ( F, x ) \mapsto x \in E _V \text{ and }\\
\mathsf{q}  :  \, & \mathrm{Gr} _{V(1),V(2)} ( V ) \ni ( F, x, \psi_1, \psi_2) \mapsto ( \psi_1 ( x \!\!\! \mod F ), \psi_2 ( x \! \mid _F ) ) \in E_{V(1)} \oplus E_{V(2)}.
\end{align*}
Notice that $\vartheta$ and $\mathsf{q}$ are smooth of relative dimensions $\dim G_{V(1)} + \dim G _{V(2)}$ and $\frac{1}{2} ( \dim G_{V} + \dim G_{V(1)} + \dim G_{V(2)} ) + \sum _{h \in \Omega} d _1 ( h' ) d _2 ( h'' )$, respectively. The map $\mathsf{p}$ is projective. We set $N_{\beta_1,\beta_2} ^{\beta} := \frac{1}{2} ( \dim G_{V} - \dim G_{V(1)} - \dim G_{V(2)} ) + \sum _{h \in \Omega} d _1 ( h' ) d _2 ( h'' )$. For $G_{V(i)}$-equivariant constructible sheaves $\mathcal F_i$ on $E_{V(i)}$ for $i = 1,2$, we define their convolution products as
$$\mathcal F_1 \odot \mathcal F_2 := \mathsf{p} _! \mathcal F _{12} [N_{\beta_1,\beta_2} ^{\beta}], \text{ where } \vartheta^* \mathcal F_{12} \cong \mathsf{q} ^* ( \mathcal F_1 \boxtimes \mathcal F_2 ) \text{ in } D^b _{G_V} ( \mathrm{Gr}  _{V(1),V(2)} ( V ) ).$$

We return to the proof of Theorem \ref{induction}. Let us fix objects $\mathbb C _{b_1, b_2}, \mathcal L _{\mathbf m^1, \mathbf m ^2}$ of $D _{G_V} ^b ( \mathrm{Gr} _{\beta_1, \beta_2} ( V ) )$  ($\mathbf m^1 \in Y ^{\beta_1}$ and $\mathbf m^2 \in Y ^{\beta_2}$) so that we have isomorphisms 
$$\vartheta ^* ( \mathbb C _{b_1, b_2} ) \cong \mathsf{q} ^* ( \mathbb C _{b_1} \boxtimes \mathbb C _{b_2} ) [N_{\beta_1, \beta_2}^{\beta}] \text{ and } \vartheta ^* \mathcal L _{\mathbf m^1, \mathbf m^2} \cong \mathsf{q} ^* ( \mathcal L _{\mathbf m^1} \boxtimes \mathcal L _{\mathbf m^2} ) [N_{\beta_1, \beta_2}^{\beta}].$$

\begin{lemma}\label{preorbit}
In the above settings, we have:
\begin{enumerate}
\item the variety $\mathsf{p}^{-1} ( \mathbb O _b )$ is a single $G_V$-orbit;
\item the map $\mathsf{p} : \mathsf{p}^{-1} ( \mathbb O _b ) \to \mathbb O_b$ is a $\mathcal P$-fibration, where $\mathcal P$ is a suitable partial flag variety.
\end{enumerate}
\end{lemma}

\begin{proof}
We have $\mathtt M _b \cong \mathtt M _{b_1} \oplus \mathtt M_{b_2}$ by assumption. The condition $(\star)_0$ asserts that the image of every two inclusions $\mathtt M _{b_2} \subset \mathtt M_b$ are transformed by $\mathrm{Aut} _{\mathbb C[\Gamma]} ( \mathtt M_b )$. Here we have $\mathrm{Aut} _{\mathbb C[\Gamma]} ( \mathtt M_b ) \cong \mathsf{Stab}_{G_V} ( x _b )$ for $x_b \in \mathbb O _b ( \mathbb C )$. Therefore, $\mathsf{p} ^{-1} ( \mathbb O _{b} )$ is a single $G_V$-orbit, which is the first assertion. Since $\mathsf{p}$ is projective, we conclude that $\mathsf{p} ^{-1} ( \mathbb O _{b} ) \to \mathbb O _b$ is projective. By $(\spadesuit)_2$, the group $\mathsf{Stab}_{G_V} ( x _b )$ is connected. Let $U_b$ denote the unipotent radical of $\mathsf{Stab}_{G_V} ( x _b )$. Since we have $\mathsf{p} ^{-1} ( \mathbb O _{b} ) \cong G_V / H_b$ with $H_b \subset \mathsf{Stab}_{G_V} ( x _b )$, the fiber of $\mathsf{p}$ is isomorphic to $\mathsf{Stab}_{G_V} ( x _b ) / H_b$, that is projective. Therefore, we deduce $U_b \subset H _b$ and the inclusion
$$H_b / U_b \subset \mathsf{Stab}_{G_V} ( x _b ) / U_b$$
must be a parabolic subgroup (of a connected reductive group). Therefore, we set $\mathcal P$ to be their quotient to deduce the second part of the result.
\end{proof}

\begin{corollary}\label{D-shift}
We have
$$\mathbb C _{b_1} [\dim \mathbb O_{b_1}] \odot \mathbb C _{b_2} [\dim \mathbb O _{b_2}] \cong D [d] \boxtimes \mathbb C _{b} [\dim \mathbb O_b],$$
where $D \cong H ^{\bullet} ( \mathcal P, \mathbb C )$ by a suitable partial flag variety $\mathcal P$ with its dimension $d$.
\end{corollary}

\begin{proof}
Thanks to $(\star)_1$, we deduce that $\vartheta ( \mathsf q^{-1} ( \mathbb O_{b_1} \times \mathbb O _{b_2} ))$ is contained in a single $G_V$-orbit. This, together with Lemma \ref{preorbit}, implies that the stalk of the LHS vanishes outside of $\mathbb O_b$. In addition, every direct summand of $\mathsf{p} _* \mathbb C _{b_1,b_2} \! \mid _{\mathbb O_b}$, viewed as a shifted $G_V$-equivariant local system (which in turn follows by \cite{BBD} 5.4.5 or 6.2.5), must be a trivial local system by $(\spadesuit)_2$. Therefore, we conclude that $\mathbb C _{b_1} [\dim \mathbb O_{b_1}] \odot \mathbb C _{b_2} [\dim \mathbb O _{b_2}] \cong D' \boxtimes \mathbb C _{b} [\dim \mathbb O_b]$ with a graded vector space $D'$. The isomorphism $D' \cong H ^{\bullet} ( \mathcal P, \mathbb C )[d]$ is by Lemma \ref{preorbit} 2).
\end{proof}

We return to the proof of Theorem \ref{induction}. In the below (during this section), we freely use the notation from Corollary \ref{D-shift}.

Thanks to Corollary \ref{allp}, $\mathcal L _{\beta_1}$ and $\mathcal L _{\beta_2}$ contains $\mathsf{IC}_{b_1}$ and $\mathsf{IC} _{b_2}$, respectively. We have
$$\mathcal L _{\mathbf m^1} \odot \mathcal L _{\mathbf m^2} \cong \mathcal L _{\mathbf m^1 + \mathbf m^2}$$
by construction. Thanks to $(\star)_1$ and \cite{BBD} 5.4.5 or 6.2.5, $\mathsf{IC} _b$ appears in $\mathcal L _{\mathbf m^1 + \mathbf m^2}$ up to a grading shift if the following condition $(\diamond)$ hold:
\begin{itemize}
\item[$(\diamond)$] $\mathsf{IC}_{b_i}$ appears in $\mathcal L _{\mathbf m ^i}$ for $i = 1,2$.
\end{itemize}

We set $\mathbf m := \mathbf m^1 + \mathbf m ^2$. Let $x_b \in \mathbb O_b$ be a point and let $i _b : \{ x_b \} \hookrightarrow E_V$ be the inclusion.

\begin{lemma}\label{indgen}
Assume that $(\diamond)$ holds. Then, the subspace
\begin{align*}
\imath _b^! \mathcal{E}xt ^{\bullet} _{D^b ( E_V )} ( \mathbb C _{b_1} \odot \mathbb C _{b_2}, \mathcal L _{\mathbf m^1 + \mathbf m^2} ) \subset & \, \imath _b^! \mathcal{E}xt ^{\bullet} _{D^b ( E_V )} ( \mathbb C _{b_1} \odot \mathbb C _{b_2}, \mathcal L _{\beta} )\\
\cong \mathcal{E}xt ^{\bullet} _{D^b ( \mathrm{pt} )} ( \imath _b^* ( \mathbb C _{b_1} \odot \mathbb C _{b_2} ), \imath _b^! \mathcal L _{\beta} ) \cong  & \, D^* [ - d ] \boxtimes K _{b} \left< 2 \dim \mathbb O_b \right>
\end{align*}
is a generating subspace as an $R_{\beta}$-module.
\end{lemma}

\begin{proof}
The isomorphism parts of the assertion follow by \cite{KS} 3.1.13 and Corollary \ref{D-shift}. By $(\diamond)$ and $(\star)_1$, we conclude that $\mathcal L _{b_1, b_2}$ contains an irreducible perverse sheaf supported on $\mathrm{Supp} \, \mathbb C_{b_1, b_2}$. Thanks to \cite{BBD} 5.4.5 or 6.2.5, we conclude that $\mathcal L _{\mathbf m^1} \odot \mathcal L _{\mathbf m^2}$ contains $\mathsf{IC} _{b}$. Therefore, the head $L_b$ of $K_b$ satisfies $e ( \mathbf m ) L_b \neq \{ 0 \}$, which proves the assertion.
\end{proof}

We set $\mathbb O _{b} ^{\uparrow} \subset E _{V}$ to be the union of $G_{V}$-orbits which contains $\mathbb O _{b}$ in its closure. Let $j _{b} ^{\uparrow} : \mathbb O _{b} ^{\uparrow} \hookrightarrow E _{V}$ be its inclusion.

\begin{proposition}\label{reduce}
We have a canonical isomorphism
$$\mathcal{E}xt ^{\bullet} _{D^b ( E_V )} ( \mathbb C _{b_1} \odot \mathbb C _{b_2}, \mathcal L _{\mathbf m^1 + \mathbf m^2} ) \cong \mathsf{p}_* \mathcal{E}xt ^{\bullet} _{D^b (\mathrm{Gr} _{\beta_1, \beta_2} (V))} ( \mathbb C_{b_1,b_2}, D^* \boxtimes \mathcal L _{\mathbf m^1, \mathbf m^2} )$$
in the bounded derived category of constructible sheaves on $E_V$.
\end{proposition}

\begin{proof}
During this proof, we repeatedly use the local form of the Verdier duality (see e.g. \cite{KS} 3.1.10, or \cite{SGA4} Expos\'e XVIII 3.1.10). We have
$$\mathcal{E}xt ^{\bullet} _{D^b ( E_V )} ( \mathbb C _{b_1} \odot \mathbb C _{b_2}, \mathcal L _{\mathbf m^1 + \mathbf m^2} ) \cong \mathsf{p}_* \mathcal{E}xt ^{\bullet} _{D^b ( \mathrm{Gr} _{\beta_1, \beta _2} ( V ) )} ( \mathbb C _{b_1, b_2}, \mathsf{p} ^! \mathcal L _{\mathbf m^1 + \mathbf m^2} ).$$
Consider the Cartesian diagram
$$\xymatrix{\mathrm{Gr} _{\beta_1, \beta _2} ( V ) \ar[r]^{\hskip 5mm \mathsf{p}} & E _V \\\mathcal G \ar[r]^{p}\ar@{^{(}->}[u] _{\jmath _b ^{\uparrow}} & \mathbb O _b ^{\uparrow} \ar@{^{(}->}[u] _{j_b ^{\uparrow}}}.$$
Note that $\jmath _b ^{\uparrow}$ is an open embedding since $\mathsf{p}$ is continuous. It follows that
\begin{align*}
& \mathsf{p}_* \mathcal{E}xt ^{\bullet} _{D^b ( \mathrm{Gr} _{\beta_1, \beta _2} ( V ) )} ( \mathbb C _{b_1, b_2}, \mathsf{p} ^! \mathcal L _{\mathbf m^1 + \mathbf m^2} ) & \\
& \cong (j _{b} ^{\uparrow}) _* p_* \mathcal{E}xt ^{\bullet} _{D^b ( \mathcal{G} )} ( (\jmath _b ^{\uparrow}) ^* \mathbb C _{b_1, b_2}, (\jmath _b ^{\uparrow}) ^! \mathsf{p} ^! \mathcal L _{\mathbf m^1 + \mathbf m^2} ) & (\mathbb C_{b_1, b_2} \cong (\jmath _b ^{\uparrow})_! (\jmath _b ^{\uparrow}) ^* \mathbb C_{b_1, b_2})\\
& \cong (j _{b} ^{\uparrow}) _* p_* \mathcal{E}xt ^{\bullet} _{D^b ( \mathcal{G} )} ( (\jmath _b ^{\uparrow}) ^* \mathbb C _{b_1, b_2}, p^! ( j _b ^{\uparrow}) ^! \mathcal L _{\mathbf m^1 + \mathbf m^2} ) & (j _b ^{\uparrow} \circ p = \mathsf{p} \circ \jmath _b ^{\uparrow}).
\end{align*}
In addition, $(\jmath _b ^{\uparrow}) ^* \mathbb C _{b_1, b_2}$ is a local system supported on the closed $G_V$-orbit $\mathcal O _b$ of $\mathcal G$. Let us denote by $\jmath _b : \mathcal O _b \hookrightarrow \mathcal G$ the inclusion. We have $(\jmath _b ^{\uparrow}) ^* \mathbb C _{b_1, b_2} \cong ( \jmath _b )_! \mathbb C [\dim \mathcal O _b]$. Thus, we deduce
\begin{align*}
& (j _{b} ^{\uparrow}) _* p_* \mathcal{E}xt ^{\bullet} _{D^b (\mathcal G)} ( (\jmath _b ^{\uparrow}) ^* \mathbb C _{b_1, b_2}, p^! ( j _b ^{\uparrow}) ^! \mathcal L _{\mathbf m^1 + \mathbf m^2} ) & \\
& \cong (j _{b} ^{\uparrow}) _* p'_* \mathcal{E}xt ^{\bullet} _{D^b (\mathcal O_b)} ( \mathbb C [\dim \mathcal O _b], \jmath_b ^! p^! ( j _b ^{\uparrow}) ^! \mathcal L _{\mathbf m^1 + \mathbf m^2} ) & ((\jmath _b ^{\uparrow}) ^* \mathbb C _{b_1, b_2} \cong ( \jmath _b )_! \mathbb C [\bullet])\\
& \cong (j _{b} ^{\uparrow}) _* p'_*\mathcal{E}xt ^{\bullet} _{D^b (\mathcal O_b)} ( \mathbb C [\dim \mathcal O _b],  D^* \boxtimes \jmath_b ^! ( \jmath_b ^{\uparrow} ) ^! \mathcal L _{\mathbf m^1, \mathbf m^2} ) & (\text{Corollary \ref{D-shift}})\\
& \cong \mathsf{p}_* \mathcal{E}xt ^{\bullet} _{D^b (\mathrm{Gr} _{\beta_1, \beta_2} (V))} ( \mathbb C_{b_1,b_2}, D^* \boxtimes \mathcal L _{\mathbf m^1, \mathbf m^2} ), &
\end{align*}
where $p' : \mathcal O _b \to \mathbb O _b ^{\uparrow}$ is the restriction of $p$. Since all the maps are canonically defined, composing all the isomorphisms yield the result.
\end{proof}

We return to the proof of Theorem \ref{induction}. Taking account into the fact $\mathsf{p} ^{-1} ( x_b ) \cong \mathcal P$, we have an isomorphism
$$D^* \left< d \right> \boxtimes e ( \mathbf m ^1 + \mathbf m ^2 ) K _b \cong \, \mathbb H ^{\bullet}i_b^! \mathcal{E}xt ^{\bullet} _{D^b ( E_V )} ( \mathbb C _{b_1} \odot \mathbb C _{b_2}, \mathcal L _{\mathbf m^1 + \mathbf m^2} )[2 \dim \mathbb O _b]$$
and a spectral sequence arising from the base change (applied to $i_b$ and $\mathsf{p}$)
\begin{align*}
E_2 := D^* \otimes H ^{\bullet} ( \mathcal P ) & \otimes \left( e ( \mathbf m ^1 ) K _{b_1} \boxtimes e ( \mathbf m ^2 ) K_{b_2} \right)\\
& \Rightarrow \mathbb H ^{\bullet} i_b^! \mathsf{p}_* \mathcal{E}xt ^{\bullet} _{D^b (\mathrm{Gr} _{\beta_1, \beta_2} (V))} ( \mathbb C_{b_1,b_2}, D^* \boxtimes \mathcal L _{\mathbf m^1, \mathbf m^2} ) [2 \dim \mathbb O _b],
\end{align*}
where we used the fact that $\dim \mathsf{p}^{-1} ( \mathbb O_b ) - \dim \mathsf{p} ^{-1} ( x_b ) = \dim \mathbb O_b$ in the degree shift of the second spectral sequence. Here the modules $K _{b_1}, K_{b_2}$, and $K_b$ are pure of weight $0$ by \cite{Lu} 10.6 (see the proof of Proposition \ref{quiverSC} for a bit precise account). By Lemma \ref{preorbit} 2), we deduce that $H^{\bullet} ( \mathcal P )$ is also pure. Therefore, the spectral sequence $E_2$ degenerates at the $E_2$-stage. By factoring out the effect of $D^*$, we conclude that
$$e ( \mathbf m ^1 + \mathbf m ^2 ) K_b \cong H ^{\bullet} ( \mathcal P ) \boxtimes \left( e ( \mathbf m^1 ) K _{b_1} \boxtimes e ( \mathbf m ^2 ) K_{b_2} \right)\left< -d \right>.$$
This induces an inclusion as $R_{\mathbf m^1, \mathbf m^1} \boxtimes R_{\mathbf m^2, \mathbf m^2}$-modules
$$\varphi _{\mathbf m^1, \mathbf m^2} : \left( e ( \mathbf m^1 ) K _{b_1} \boxtimes e ( \mathbf m ^2 ) K_{b_2} \right)\left< d \right> \hookrightarrow e ( \mathbf m ^1 + \mathbf m ^2 ) K_b.$$
The module $e ( \mathbf m^1 + \mathbf m^2 ) K_b$ admits an $R_{\mathbf m^1 + \mathbf m^2,\mathbf m^1 + \mathbf m^2}$-module structure with simple head thanks to Theorem \ref{KS} 3). This extends the $R_{\mathbf m^1, \mathbf m^1} \boxtimes R_{\mathbf m^2, \mathbf m^2}$-module structure. Recall that for each $i=\emptyset,1,2$, the simple head of $K_{b_{i}}$ as an irreducible $R_{\beta _i}$-module is realized as the coefficient vector space of $\mathsf{IC} _{b_i}$ inside $\mathcal L _{\beta _i}$ (see \S 1), and its weight $e ( \mathbf m^i )$-part is that of $\mathcal L _{\mathbf m^i}$ (see \S 2). (Note that this sheaf-theoretic interpretation gives a splitting of $L_{b_i}$ to $K_{b_i}$ as vector spaces for each $i = \emptyset,1,2$.) By \cite{BBD} 5.4.5 or 6.2.5 and Corollary \ref{D-shift}, the complex $H ^{\bullet} (\mathcal P) [d] \boxtimes \mathsf{IC} _b$ is a direct summand of $\mathsf{IC} _{b_1} \odot \mathsf{IC} _{b_2}$. Therefore, the above interpretation implies that the unique simple quotients $L_{b_1}$ and $L_{b_2}$ of $K_{b_1}$ and $K_{b_2}$ satisfy
$$\varphi _{\mathbf m^1, \mathbf m^2} ( H^{\bullet} ( \mathcal P ) \otimes \left( e ( \mathbf m^1 ) L_{b_1} \boxtimes e ( \mathbf m^2 ) L_{b_2} \right) ) \left< - d \right> \subset e ( \mathbf m ) L_{b} \subset e ( \mathbf m ) K_{b}$$
as vector subspaces, where $L_b$ is the simple head of $K_{b}$. In addition, these inclusions are non-zero if $\mathbf m^1$ and $\mathbf m^2$ satisfies $(\diamond)$. Since we can choose $\mathbf m^1$ and $\mathbf m^2$ so that $(\diamond)$ is satisfied, we have a surjective map of graded $R_{\beta}$-modules:
$$K _{b_1} \star K _{b_2} \left< d \right> \longrightarrow \!\!\!\!\! \rightarrow K _b.$$

\begin{lemma}\label{ind-dim-est} In the above settings, we have
$$\dim K _b = \dim \left( K _{b_1} \star K _{b_2} \right).$$
\end{lemma}

\begin{proof}
In this proof, $i$ denotes either $\emptyset, 1$, or $2$. Let us choose a point $x_{b_i} \in \mathbb O _{b_i} \subset E_{V (i)}$. Let $T_i$ be a maximal torus of $\mathsf{Stab} _{G_{V(i)}} x _{b_i}$. Choose $\mathbf m ^i \in Y ^{\beta _i}$. Thanks to the purity of each module (Lusztig \cite{Lu} 10.6), we deduce that the spectral sequence
$$H ^{\bullet} _{T_i} ( \mathrm{pt} ) \otimes H _{\bullet} ( \pi _{\mathbf m^i}^{-1} ( x_{b_i} ) ) \Rightarrow H _{\bullet} ^{T_i} ( \pi _{\mathbf m^i}^{-1} ( x_{b_i} ) )$$
degenerates at the $E_2$-stage. Here the RHS have the same $H ^{\bullet} _{T_i} ( \mathrm{pt} )$-rank as that of $H _{\bullet}^{T_i} ( \pi _{\mathbf m^i}^{-1} ( x_b ) ^{T_i} ) $. Therefore, we have
$$\dim H _{\bullet} ( \pi _{\mathbf m^i}^{-1} ( x_b ) ) = \dim H _{\bullet} ( \pi _{\mathbf m^i}^{-1} ( x_{b_i} ) ^{T_i} ).$$
By Theorem \ref{PBW}, we deduce that $R_{\beta}$ is a free $R_{\beta_1} \boxtimes R_{\beta _2}$-module of rank $\frac{n!}{n_1! n_2!}$. Hence, it is enough to show
\begin{align*}
\sum _{\mathbf m \in Y ^{\beta}} \dim H _{\bullet} ( \pi _{\mathbf m}^{-1} ( x_b ) ^{T} ) & \\
= \frac{n!}{n_1! n_2!} & \sum _{\scriptsize \begin{matrix}\mathbf m ^1 \in Y ^{\beta_1}\\ \mathbf m^2 \in Y ^{\beta_2}\end{matrix}} ( \dim H _{\bullet} ( \pi _{\mathbf m ^1}^{-1} ( x_{b_1} ) ^{T_1} ) ) ( \dim H _{\bullet} ( \pi _{\mathbf m ^2}^{-1} ( x_{b_2} ) ^{T_2} ) ).
\end{align*}
This follows by a simple counting since $E_{V(i)} ^{T _i}$ decomposes into the product of varieties corresponding to each indecomposable module.
\end{proof}

We return to the proof of Theorem \ref{induction}. Lemma \ref{ind-dim-est} asserts that
$$K _{b_1} \star K _{b_2} \left< d \right> \cong K _b$$
as graded $R_{\beta}$-modules. This completes the proof of Theorem \ref{induction} except for the last assertion. The last assertion follows since the assumption implies that $\mathsf p_{b_1,b_2}$ is birational onto its image, and hence $d = 0$.

\section{Characterization of the PBW bases}\label{Chap:charPBW}
Keep the setting of the previous section. For a reduced expression $\mathbf i$ of $w_0$ and a sequence of non-negative integers $\mathbf c := ( c_{1}, c_{2}, \ldots, c _{\ell} ) \in \mathbb Z _{\ge 0} ^{\ell}$, we call the pair $( \mathbf i, \mathbf c )$ a Lusztig datum, and we call $\mathbf c$ an $\mathbf i$-Lusztig datum. For a Lusztig datum $( \mathbf i, \mathbf c )$, we define
$$\mathsf{wt} ( \mathbf i, \mathbf c ) := \sum _{k = 1} ^{\ell} c_k \gamma ^{(k)} _{\mathbf i}, \hskip 3mm \text{ where } \hskip 3mm \gamma ^{(k)} _{\mathbf i} := s_{i_{1}} \cdots s_{i_{k-1}} \alpha _{i_k}.$$
For two $\mathbf i$-Lusztig data $\mathbf c$ and $\mathbf c'$, we define $\mathbf c < _{\mathbf i} \mathbf c'$ as: There exists $0 \le k < \ell$ so that
$$c_1 = c_1', c_2 = c_2', \ldots, c_k = c_k' \hskip 2mm \text{ and } \hskip 2mm c_{k+1} > c_{k+1}'.$$
Associated to each Lusztig datum $( \mathbf i, \mathbf c )$, we define the lower PBW-module $\widetilde{E} ^{\mathbf i} _{\mathbf c}$ as:
\begin{equation}
\widetilde{E} ^{\mathbf i} _{\mathbf c} := P_{c_1 i_1} \star \mathbb T_{i_1} \left( P _{c_2 i_2} \star \mathbb T_{i_2} \left( P _{c_3 i_3} \star \cdots \mathbb T_{i_{\ell-1}} P_{c_{\ell} i_{\ell}} \right) \cdots \right).\label{PBW-01}
\end{equation}
Similarly, we define the corresponding upper PBW-module $E ^{\mathbf i} _{\mathbf c}$ as:
\begin{equation}
E ^{\mathbf i} _{\mathbf c} := L_{c_1 i_1} \star \mathbb T_{i_1} \left( L _{c_2 i_2} \star \mathbb T_{i_2} \left( L _{c_3 i_3} \star \cdots \mathbb T_{i_{\ell-1}} L _{c_{\ell} i_{\ell}} \right) \cdots \right).\label{PBW-02}
\end{equation}
By construction, it is clear that $E ^{\mathbf i} _{\mathbf c}$ is a quotient of $\widetilde{E} ^{\mathbf i} _{\mathbf c}$.

\begin{remark}
The modules $E ^{\mathbf i} _{\mathbf c}$ and $\widetilde{E} ^{\mathbf i} _{\mathbf c}$ are identified with $\widetilde{K} ^{\Omega} _{b}$ and $K ^{\Omega} _{b}$ from \S 2 when the reduced expression $\mathbf i$ is adapted to $\Omega$ (cf. Corollary \ref{PBW-Kos}). In addition, we have $P_{c i} = \widetilde{K} _{c i}$ and $L_{c i} = K _{c i}$ for each $c \in \mathbb Z_{\ge 0}$ and $i \in I$ by Example \ref{one-root}.
\end{remark}

\begin{lemma}\label{bPBW}
For each Lusztig datum $( \mathbf i, \mathbf c )$, we have:
\begin{enumerate}
\item $\widetilde{E} ^{\mathbf i} _{\mathbf c}$ and $E ^{\mathbf i} _{\mathbf c}$ are $R_{\mathsf{wt} ( \mathbf i, \mathbf c )}$-modules;
\item $\widetilde{E} ^{\mathbf i} _{\mathbf c}$ and $E ^{\mathbf i} _{\mathbf c}$ are modules with simple heads if they are non-zero;
\item $\widetilde{E} ^{\mathbf i} _{\mathbf c} \neq \{ 0 \}$ if and only if $E ^{\mathbf i} _{\mathbf c} \neq \{ 0 \}$.
\end{enumerate}
\end{lemma}

\begin{proof}
Since $\mathbb T_i$ is a functor sending an $R_{\beta}$-module to an $R_{s_i \beta}$-module (possibly zero), the first assertion is immediate. The functor $\mathbb T_i$ also preserves the simple head property (provided if it does not annihilate the whole module) by construction. Therefore, we apply Lemma \ref{indpro} repeatedly to deduce the simple head property of $\widetilde{E} ^{\mathbf i} _{\mathbf c}$ and $E ^{\mathbf i} _{\mathbf c}$ from that of $P_{c _{k} i_{k}}$ ($1 \le k \le \ell$), which is the second assertion. By construction, $E ^{\mathbf i} _{\mathbf c}$ contains the head of $\widetilde{E} ^{\mathbf i} _{\mathbf c}$, and hence the third assertion.
\end{proof}

\begin{theorem}[Lusztig \cite{Lu}]\label{adapted}
Assume that the reduced expression $\mathbf i$ is adapted to $\Omega$. Then, we have $E ^{\mathbf i} _{\mathbf c} \neq \{ 0 \}$ for every $\mathbf i$-Lusztig datum. Moreover, the set of $\mathbf i$-Lusztig data is in bijection with $B ( \infty )$ as:
$$\mathbf c \mapsto \mathsf{hd} \, E ^{\mathbf i} _{\mathbf c} \cong L _b \hskip 3mm \text{ for } \hskip 3mm b \in B ( \infty ).$$
\end{theorem}

\begin{proof}
Since $\mathbf i$ is adapted, we deduce that a module $\mathbb T_{i_1} \mathbb T_{i_2} \cdots \mathbb T_{i_{k-1}} L _{c_k i_k}$ is simple and it corresponds to an indecomposable $\mathbb C [\Gamma]$-module $\mathtt M_{(k)}$ with $\underline{\dim} \, \mathtt M_{(k)} = \gamma _{\mathbf i} ^{(k)}$ (\cite{Lu} 4.7). We apply Corollary \ref{i-ind} and Theorem \ref{SRF} 1) repeatedly to construct a module with its simple head corresponding to the quiver representation $\mathtt M _{(1)}^{\oplus c_1} \oplus \cdots \oplus \mathtt M _{(\ell)} ^{\oplus c_{\ell}}$. Now the Gabriel theorem yields the result.
\end{proof}

\begin{definition}[$2$-move, $3$-move, \cite{Lu} 2.3]
We say that two Lusztig data $(\mathbf i, \mathbf c)$ and $(\mathbf i', \mathbf c')$ are connected by a $2$-move if
\begin{enumerate}
\item there exists $1 \le k < \ell$ so that $i_k = i_{k+1}', i_{k+1} = i_k'$, $i_k \not\leftrightarrow i_{k+1}$, and $i_l = i_l'$ for every $l \neq k, k+1$;
\item we have $c_k = c_{k+1}', c_{k+1} = c_k'$, and $c_l = c_l'$ for every $l \neq k, k+1$.
\end{enumerate}
We say that $(\mathbf i, \mathbf c)$ and $(\mathbf i', \mathbf c')$ are connected by a $3$-move if
\begin{enumerate}
\item there exists $1 < k < \ell$ so that $i_{k-1} = i_{k+1} = i_{k}'$, $i_{k} = i'_{k-1} = i_{k+1}'$, $i_{k} \leftrightarrow i_{k+1}$, and $i_l = i_l'$ for every $l \neq k-1, k, k+1$;
\item we have $c_l = c_l'$ for every $l \neq k-1, k, k+1$, and
$$( c_{k-1}', c_{k}', c_{k+1}' ) = ( c_{k} + c_{k+1} - c_{0}, c_{0} , c_{k-1} + c_{k} - c_{0} ) \text{ for } c_{0} := \min \{ c_{k-1}, c_{k+1} \}.$$
\end{enumerate}
\end{definition}

\begin{lemma}\label{2me}
For two Lusztig data $( \mathbf i, \mathbf c)$ and $( \mathbf i', \mathbf c')$ which are connected by a $2$-move, we have $E ^{\mathbf i} _{\mathbf c} \cong E ^{\mathbf i'} _{\mathbf c'}$.
\end{lemma}

\begin{proof}
Find a unique $1 \le k < \ell$ so that $i_k = i_{k+1}' \neq i_{k+1} = i_{k}'$. We realize $\mathbb T_{i_k}$ and $\mathbb T_{i_{k+1}}$ by choosing the orientation $\Omega$ so that the both of $i_k, i_{k+1}$ are source (which is in turn possible since $i_k \not\leftrightarrow i_{k+1}$). We have $\mathbb T_{i_k} L _{c_{k+1} i_{k+1}} = L _{c_{k+1} i_{k+1}}$ and $\mathbb T_{i_{k+1}} L _{c_{k} i_{k}} = L _{c_{k} i_{k}}$ since $R_{p \alpha _{i_k} + q \alpha _{j_k}}$ is Morita equivalent to $R_{p \alpha _{i_k}} \boxtimes R_{q \alpha _{j_k}}$ for each $p,q \ge 0$ by the product decomposition of $( G_V, E_V ^{\Omega} )$. Applying Theorem \ref{SRF} 2) and 5), it suffices to prove
\begin{equation}
L_{c_k i_k} \star L _{c_{k+1} i_{k+1}} = L_{c_k i_k} \star \mathbb T_{i_k} L _{c_{k+1} i_{k+1}} \cong L_{c_{k+1} i_{k+1}} \star \mathbb T_{i_{k+1}} L _{c_{k} i_{k}} =  L _{c_{k+1} i_{k+1}}\star L_{c_k i_k} \label{2move}
\end{equation}
and $\mathbb T_{i_k} \mathbb T _{i_{k+1}} \cong \mathbb T_{i_{k+1}} \mathbb T _{i_k}$. The product decomposition of $( G_V, E_V ^{\Omega} )$ also provides (\ref{2move}) as $\star$ is the same as the external tensor product here.

We have $T_{i_k} T_{i_{k+1}} ( b ) = T_{i_{k+1}} T_{i_k} ( b )$, $\epsilon _{i_k} ( b ) =  \epsilon _{i_k} ( T_{i_{k+1}} ( b ))$, and $\epsilon _{i_{k+1}} ( b ) =  \epsilon _{i_{k+1}} ( T_{i_k} ( b ))$ by inspection. The essential image of the functor $\mathbb T_i$ (applied to $R_{s_i \beta} \mathchar`-\mathsf{gmod}$ for some $\beta \in Q^+$) is equivalent to ${}_i R_{\beta} \mathchar`-\mathsf{gmod}$ by construction. We have $e_{i_k} (1)e_{i_{k+1}} (1) = 0 = e_{i_{k+1}} (1) e_{i_k} (1)$ by definition. Therefore, we deduce that the essential image of each of the functors $\mathbb T_{i_k} \mathbb T_{i_{k+1}}$ and $\mathbb T_{i_{k+1}} \mathbb T_{i_k}$ is equivalent to the graded module category of
$$\left( R_{\beta} / ( R_{\beta} e_{i_k} (1) R_{\beta} ) \right) \otimes _{R_{\beta}} \left( R_{\beta} / ( R_{\beta} e_{i_{k+1}} (1) R_{\beta} ) \right) \cong R_{\beta} / ( R_{\beta} e R_{\beta} ) \hskip 1mm \text{ for some } \beta \in Q^+,$$
where $e := e_{i_k} (1) + e_{i_{k+1}} (1)$ is the minimal idempotent so that $e e_{i_k} (1) = e_{i_k} (1)$ and $e e_{i_{k+1}} (1) = e_{i_{k+1}} (1)$. Hence, Theorem \ref{SRF} 2) guarantees that $\mathbb T_{i_k} \mathbb T_{i_{k+1}} \cong \mathbb T_{i_{k+1}} \mathbb T_{i_k}$ as functors, which completes the proof.
\end{proof}

\begin{proposition}\label{3me-true}
Let $( \mathbf i, \mathbf c )$ and $( \mathbf i', \mathbf c')$ be two Lusztig data which are connected by a $3$-move as $( i_{k-1}, i_k, i_{k+1} ) = ( i'_{k}, i'_{k\pm 1}, i'_k)$ for some $k$. Then, we have $\mathsf{hd} \, E ^{\mathbf i} _{\mathbf c} \cong \mathsf{hd} \, E ^{\mathbf i'} _{\mathbf c'}$.
\end{proposition}

\begin{proof}
Let $i_k = j, i_{k+1} = i$.
By an explicit calculation (which reduces to the rank two case, cf. Theorem \ref{adapted}), we see that
\begin{equation}
\mathsf{hd} \, \left( L_{c_{k-1} i} \star \mathbb T_{i} L_{c_{k} j} \star \mathbb T_{i}\mathbb T _{j} L_{c_{k+1} i} \right) \cong \mathsf{hd} \, \left( L_{c_{k-1}' j} \star \mathbb T_{j} L_{c_{k}' i} \star \mathbb T_{j}\mathbb T _{i} L_{c_{k+1}' j} \right).\label{simple-move}
\end{equation}

By Lemma \ref{bPBW} 2), it suffices to show
$$\mathsf{hd} \, E ^{\mathbf i} _{\mathbf c} \cong \mathsf{hd} \, E ^{\mathbf i'} _{\mathbf c'}$$
for every $\mathbf i$-Lusztig datum $\mathbf c$ so that $c_1 = \cdots = c_{k+1} = 0$ (and its counterpart $\mathbf i'$-Lusztig datum $\mathbf c'$).

We borrow the settings from the (case of $i \leftrightarrow j$ in the) proof of Corollary \ref{j-ind}. In particular, we arrange $\Omega$ so that $j \in I$ is a sink and $i \in I$ is a sink of $s_j \Omega$. Thanks to Theorem \ref{SRF} 2) and the construction of the Saito reflection functor, the set of simple modules that is not annihilated by $\mathbb T_i^* \mathbb T_j^*$ corresponds to a representation of $\mathbb C [\Gamma]$ that does not contain a direct factor from $\{ \mathtt M_{i,j}, \mathtt M_j \}$. For such $d \in B ( \infty )$, we have a corresponding $\mathbb C [\Gamma]$-module $\mathtt M_d$.

By inspection, we deduce that the pairs $( \mathtt M_j ^{\oplus c_{k-1}}, \mathtt M_d )$ and $( \mathtt M_{i,j} ^{\oplus c_{k}} \oplus \mathtt M_j ^{\oplus c_{k+1}}, \mathtt M_d )$ satisfy the assumption of Theorem \ref{induction}. We assume
$$L \cong \mathsf{hd} \, E ^{\mathbf i} _{\mathbf d} \hskip 5mm \text{and} \hskip 5mm M \cong \mathsf{hd} \, E ^{\mathbf i} _{\mathbf c},$$
where $\mathbf d = (0, \ldots, 0, c_{k+2}, c_{k+3} , \ldots, c_{\ell} )$ by Lemma \ref{bPBW} 2). Note that
$$L \cong \mathbb T_{i_1}\mathbb T_{i_2} \cdots \mathbb T_{i_{k-1}} L' \hskip 5mm \text{and} \hskip 5mm M \cong \mathbb T_{i_1}\mathbb T_{i_2} \cdots \mathbb T_{i_{k-1}} M'$$
for the simple module $L'$ corresponding to $d$ and a simple module $M'$. We have
\begin{align*}
\mathbb T_{i_{k-2}}^* \mathbb T_{i_{k-3}}^* \cdots \mathbb T_{i_{1}}^* M & \cong \mathsf{hd} \, \left( L_{c_{k-1}i_{k-1}} \star \mathbb T_i ( L_{c_k j} \star \mathbb T_j ( L_{c_{k+1}i} \star \mathbb T_j^* \mathbb T_{i}^* L' ) ) \right)\\
& \cong \mathsf{hd} \, \left( L_{c_{k-1}i_{k-1}} \star \mathbb T_i ( L_{c_k j} \star \mathbb T_j L_{c_{k+1}i} \star \mathbb T_{i}^* L' ) \right)\\
& \cong \mathsf{hd} \, \left( L_{c_{k-1}i_{k-1}}  \star \mathbb T_i ( L_{c_k j} \star \mathbb T_j L_{c_{k+1}i} ) \star L' \right)\\
& \cong \mathsf{hd} \, \left( L_{c_{k-1}i_{k-1}}  \star \mathbb T_i L_{c_k j} \star \mathbb T_i \mathbb T_j L_{c_{k+1}i} \star L' \right)
\end{align*}
where the second and the fourth isomorphisms are by Theorem \ref{SRF} 5), and the third isomorphism is obtained from the combination of Theorem \ref{SRF} 1) and Theorem \ref{induction} applied to $( \mathtt M_{i,j} ^{\oplus c_{k}} \oplus \mathtt M_j ^{\oplus c_{k+1}}, \mathtt M_d )$ and take some quotient (that is possible since $L_{c_k j} \star \mathbb T_j L_{c_{k+1}i}$ defines a standard module arising from $\Omega$ and $\mathbb T_i^* L'$ is an unique simple quotient of the standard module corresponding to $\mathtt M_d$).

Since the same is true if we replace $\mathbf i$ with $\mathbf i'$ and swap $(i,j)$, we conclude the assertion from (\ref{simple-move}).
\end{proof}

\begin{corollary}\label{LparamS}
For Lusztig data $(\mathbf i, \mathbf c)$ and $(\mathbf i', \mathbf c')$, we have $\mathsf{hd} \, E _{\mathbf c} ^{\mathbf i} \cong \mathsf{hd} \, E _{\mathbf c'} ^{\mathbf i'}$ if and only if $(\mathbf i, \mathbf c)$ and $(\mathbf i', \mathbf c')$ are linked by a successive application of two-moves and three-moves.
\end{corollary}

\begin{proof}
Every two reduced expressions of $w_0 \in W (\Gamma_0)$ are connected by a repeated use of two moves and three moves (\cite{Lu} 2.1 (c)). Therefore, we apply Lemma \ref{2me} and Proposition \ref{3me-true} repeatedly from Theorem \ref{adapted} to deduce the assertion.
\end{proof}

\begin{corollary}\label{Lparam}
The module $E _{\mathbf c} ^{\mathbf i}$ is non-zero for every Lusztig datum $( \mathbf i, \mathbf c )$, and the map
$$\mathbf c \mapsto \mathsf{hd} \, E _{\mathbf c} ^{\mathbf i} \cong L_b \text{ for } b \in B (\infty)$$
sets up a bijection between the set of $\mathbf i$-Lusztig data and $B (\infty)$.
\end{corollary}

\begin{proof}
This is a special case of Corollary \ref{LparamS}.
\end{proof}

Thanks to Corollary \ref{Lparam}, we often write $\widetilde{E} _{b} ^{\mathbf i}$ and $E _{b} ^{\mathbf i}$ instead of $\widetilde{E} _{\mathbf c} ^{\mathbf i}$ and $E _{\mathbf c} ^{\mathbf i}$. For an $\mathbf i$-Lusztig data $\mathbf c = (c_1,c_2,\ldots,c_{\ell})$, we define $\overline{\mathbf c} := (c_{\ell},c_{\ell -1},\ldots,c_1)$.

\begin{proposition}\label{bPBW2}
For each Lusztig datum $( \mathbf i, \mathbf c )$, it holds:
\begin{enumerate}
\item we have surjections as graded $R_{\mathsf{wt} ( \mathbf i, \mathbf c )}$-modules:
\begin{align*}
\widetilde{E} ^{\mathbf i} _{\mathbf c} & \longrightarrow \!\!\!\!\! \rightarrow P_{c_1 i_1} \star ( \mathbb T_{i_1} P _{c_2 i_2} ) \star ( \mathbb T_{i_1} \mathbb T_{i_2} P _{c_3 i_3} ) \star \cdots \star ( \mathbb T_{i_1} \cdots \mathbb T_{i_{\ell-1}} P_{c_{\ell} i_{\ell}} )\\
E ^{\mathbf i} _{\mathbf c} & \longrightarrow \!\!\!\!\! \rightarrow L_{c_1 i_1} \star ( \mathbb T_{i_1} L _{c_2 i_2} ) \star ( \mathbb T_{i_1} \mathbb T_{i_2} L _{c_3 i_3} ) \star \cdots \star ( \mathbb T_{i_1} \cdots \mathbb T_{i_{\ell-1}} L_{c_{\ell} i_{\ell}} ):
\end{align*}
\item we have $[ E ^{\mathbf i} _{\mathbf c} : \mathsf{hd} \, E ^{\mathbf i} _{\mathbf c} ] = 1$, and $[ E ^{\mathbf i} _{\mathbf c} : \mathsf{hd} \, E ^{\mathbf i} _{\mathbf c'} ] = 0$ if $\mathbf c \not< _{\mathbf i} \mathbf c'$ or $\overline{\mathbf c} \not< _{\mathbf i} \overline{\mathbf c'}$;
\item the module $\widetilde{E} ^{\mathbf i} _{\mathbf c}$ is a successive self-extension of $E ^{\mathbf i} _{\mathbf c}$.
\end{enumerate}
\end{proposition}

\begin{proof}
The first assertion follows by a repeated use of Lemma \ref{premon} to the definitions.

For the second assertion, let us find $b \in B ( \infty )$ so that $L_b = \mathsf{hd} \, E ^{\mathbf i} _{\mathbf c'}$ by Corollary \ref{Lparam}. 

For each $1 \le k \le \ell$, we set
$$\mathbf c [k] := ( \overbrace{0,\ldots,0}^{k},c_{k+1},c_{k+2},\ldots, c_{\ell}).$$
Lemma \ref{wt-e} and the definition of $\mathbb T_i$ imply that $c_1 \ge c'_1$ if $[ E ^{\mathbf i} _{\mathbf c} : L_b ] = 0$. Moreover, $c_1 = c'_1$ implies that $e_i ( c_1 ) L_b$ is a quotient of $E ^{\mathbf i} _{\mathbf c[1]}$ by Theorem \ref{crys}. In particular, $\mathbb T_{i_1}^* ( e_i ( c_1 ) L_b )$ is an irreducible constituent of $\mathbb T_{i_1}^* E ^{\mathbf i} _{\mathbf c[1]}$. If we have $c_1 = c_1',\ldots,c_k = c_k'$, then we apply Theorem \ref{SRF} 2) repeatedly to obtain
$$\mathbb T_{i_k}^* e_{i_k} (c_k) \cdots \mathbb T_{i_2}^* e_{i_2} (c_2)  \mathbb T_{i_1}^* e_{i_1} (c_1) L_b \neq \{ 0 \}$$
and it is an irreducible constituent of $\mathbb T_{i_k}^*e_{i_k} (c_k) \cdots \mathbb T_{i_2}^* e_{i_2} (c_2)  \mathbb T_{i_1}^* e_{i_1} (c_1) E ^{\mathbf i} _{\mathbf c}$. Therefore, Lemma \ref{wt-e} implies $c_k \ge c_k'$ and we deduce the condition $\mathbf c \not< _{\mathbf i} \mathbf c'$ when $[ E ^{\mathbf i} _{\mathbf c} : L_b ] = 0$. In view of Theorem \ref{crys} and Theorem \ref{SRF} 2), this argument also asserts $[ E ^{\mathbf i} _{\mathbf c} : \mathsf{hd} \, E ^{\mathbf i} _{\mathbf c} ] = 1$.

We set $\{j_1,\ldots,j_{\ell}\} \in I^{\ell}$ by $\alpha_{j_k} := - w _0 \alpha_{i_{\ell - k + 1}}$. By Corollary \ref{Lparam}, the module
$$\mathbb T_{i_1} \cdots \mathbb T_{i_{k-1}} L_{c i_{k}}$$
is simple and non-zero for every $1 \le k \le \ell$ and $c \in \mathbb Z_{> 0}$. Moreover, we have
$$
\mathbb T_{i_1} \cdots \mathbb T_{i_{k-1}} L_{c i_{k}} \cong \mathbb T_{j_1} ^* \cdots \mathbb T_{j_{\ell - k + 2}} ^* L_{c i_{\ell - k + 1}}
$$
since $\{ j_{\ell - k + 2}, j_{\ell - k + 1}, \ldots, j_1, i_1,\ldots, i_k \}$ defines a reduced expression of $w_0$. It follows that $e _{j_1}^* ( 1 ) \mathbb T_{i_{1}} \cdots \mathbb T_{i_{k-1}} L_{c i_{k}} = \{ 0 \}$ except for $k = \ell$. Since every two reduced expressions of $s_{j_1} w_0$ are connected by a successive applications of two moves and three moves, we can choose an adapted reduced expression $\mathbf i'$ of $w_0$ ending on $i_{\ell}$ that is connected to $\mathbf i$ without changing the last entry. For $\mathbf i'$, the surjection in the first assertion must be isomorphism by a repeated application of Theorem \ref{induction} (see also the proof of Theorem \ref{adapted} and Theorem \ref{SRF} 1)). Moreover, Lemma \ref{wt-e} asserts $[ E ^{\mathbf i} _{\mathbf c}  : L_b ] \neq 0$ only if $c_{\ell} \ge c'_{\ell}$ by examining $\epsilon^*_{j_1}$ (for $\mathbf i'$).

By the same arguments as in the proof of Lemma \ref{2me} and Proposition \ref{3me-true}, we deduce that each two move and three move (that fixes the last entry $i_{\ell}$) preserves the property $[ E ^{\mathbf i} _{\mathbf c} : L_b ] \neq 0$ only if $c_{\ell} \ge c'_{\ell}$. Therefore, we deduce that $[ E ^{\mathbf i} _{\mathbf c}  : L_b ] \neq 0$ only if $c_{\ell} \ge c'_{\ell}$. From this, we deduce the condition $\overline{\mathbf c} \not< _{\mathbf i} \overline{\mathbf c'}$ when $[ E ^{\mathbf i} _{\mathbf c} : L_b ] = 0$ by induction using the the reduced expression of the form $\{j_k,\ldots, j_1,i_1,\ldots,i_{\ell - k + 1}\}$ ($1 \le k \le \ell$), $\epsilon^*_i$, $e_i^*$, $\mathbb T_i$ instead of $\mathbf i$, $\epsilon_i$, $e_i$, $\mathbb T_i^*$. This complete the proof of the second assertion.

We prove the third assertion. The module $P_{c i_{k}}$ is the maximal self-extension of $L_{c i_{k}}$ (for every $1 \le k \le \ell$ and $c \in \mathbb Z_{> 0}$) by inspection. The second assertion guarantees that the $\mathbb T$'s appearing in the definition $E ^{\mathbf i} _{\mathbf c}$ does not annihilate every irreducible constituent. Therefore, $\widetilde{E} ^{\mathbf i} _{\mathbf c}$ is a self-extension of $E ^{\mathbf i} _{\mathbf c}$ as required.
\end{proof}

Since $R_{\beta}$ is a finitely generated algebra free over a polynomial ring (by Theorem \ref{PBW}) with finite global dimension (by Theorem \ref{Kashiwara}), it follows that
$$[ M : L_{b} ], \hskip 1mm \left< M, N \right> _{\mathsf{gEP}} \in \mathbb Z (\!(t)\!),\hskip 2mm \text{ and } \hskip 2mm \mathsf{gch} \, M \in \bigoplus _{b \in B(\infty) _{\beta}} \mathbb Z (\!( t )\!) [L_b]$$
for $M,N \in R_{\beta} \mathchar`-\mathsf{gmod}$. (See (\ref{gepdef}) and Theorem \ref{VV} for $\left< \bullet, \bullet \right> _{\mathsf{gEP}}$.)

\begin{corollary}\label{R-indep}
Fix a reduced expression $\mathbf i$ and let $\beta \in Q_+$. Then, two sets $\{ \mathsf{gch} \, \widetilde{E} ^{\mathbf i} _{b} \} _{b \in B(\infty) _{\beta}}$ and $\{ \mathsf{gch} \, E ^{\mathbf i} _{b} \} _{b \in B(\infty) _{\beta}}$ are $\mathbb Z (\!( t )\!)$-bases of $\bigoplus _{b \in B(\infty) _{\beta}} \mathbb Z (\!( t )\!) [L_b]$, respectively.
\end{corollary}

\begin{proof}
Thanks to Corollary \ref{Lparam} and Lemma \ref{bPBW} 3), we deduce
$$\mathsf{gch} \, \widetilde{E} ^{\mathbf i} _{b} \left< c_b \right> \in [ L _b ] + \bigoplus _{b' \in B(\infty) _{\beta}} t \mathbb Z [\![ t ]\!] [L_{b'}] \hskip 3mm \text{ for some } \hskip 3mm c_b \in \mathbb Z.$$
This is enough to see the first assertion. (In fact, we can show $c_b = 0$ by a standard argument, or a consequence of Theorem \ref{main} 3).) The second assertion is similar.
\end{proof}

Thanks to Corollary \ref{R-indep}, we define $[M : \widetilde{E} ^{\mathbf i} _{b}], [M : E ^{\mathbf i} _{b}] \in \mathbb Z (\!( t )\!)$ for every $M \in R _{\beta} \mathchar`-\mathsf{gmod}$ as:
$$\mathsf{gch} \, M = \sum _{b \in B(\infty) _{\beta}} [M : \widetilde{E} ^{\mathbf i} _{b}] \,\mathsf{gch} \, \widetilde{E} ^{\mathbf i} _{b} \hskip 2mm \text{ and } \hskip 2mm \mathsf{gch} \, M = \sum _{b \in B(\infty) _{\beta}} [M : E ^{\mathbf i} _{b}] \,\mathsf{gch} \, E ^{\mathbf i} _{b}.$$

For a reduced expression $\mathbf i = (i_1, \ldots, i_{\ell})$ of $w_0$, we have a unique reduced expression of the form $\mathbf i^{\#} := ( i_2,i_3,\ldots,i_{\ell}, i_1' )$. (Namely $s_{i_1'} := w_0 s_{i_1} w_0 ^{-1}$.)

\begin{theorem}\label{main}
Fix a reduced expression $\mathbf i$ and $\beta \in Q^+$. We have:
\begin{enumerate}
\item For every $b < _{\mathbf i} b'$, it holds $\mathrm{ext} ^{\bullet} _{R_{\beta}} ( \widetilde{E} ^{\mathbf i} _{b}, \widetilde{E} ^{\mathbf i} _{b'} ) = \{ 0 \}$;
\item For each $b \in B ( \infty )_{\beta}$, we have
$$\mathrm{ext} ^{\bullet} _{R_{\beta}} ( \widetilde{E} ^{\mathbf i} _{b}, E ^{\mathbf i} _{b} ) = \mathrm{hom} _{R_{\beta}} ( \widetilde{E} ^{\mathbf i} _{b}, E ^{\mathbf i} _{b} ) \cong \mathbb C;$$
\item For each $b \in B ( \infty )_{\beta}$, we have
$$[ E^{\mathbf i} _{b} : L_{b'} ] = \begin{cases} 0 & (b \not\le_{\mathbf i} b')\\ 1 & (b=b') \end{cases} \hskip 3mm \text{ and } \hskip 3mm  [ \widetilde{E}^{\mathbf i} _{b} : L_{b'} ] = 0 \hskip 3mm (b \not\le_{\mathbf i} b');$$
\item For every $b \le _{\mathbf i} b'$, it holds
$$\mathrm{ext} ^{\bullet} _{R_{\beta}} ( \widetilde{E} ^{\mathbf i} _{b}, ( E ^{\mathbf i} _{b'} )^* ) \cong \mathrm{hom} _{R_{\beta}} ( \widetilde{E} ^{\mathbf i} _{b}, ( E ^{\mathbf i} _{b'} )^* ) \cong \mathbb C ^{\oplus \delta _{b,b'}}.$$
\end{enumerate}
\end{theorem}

\begin{proof}
We fix two elements $b < _{\mathbf i} b' \in B (\infty)_{\beta}$ which correspond to $\mathbf i$-Lusztig data $\mathbf c$ and $\mathbf c'$, respectively. Let $m$ be the smallest number so that $c_m \neq 0$. (Note that $c'_1 = \cdots = c'_{m-1} = 0$.) We note that the third assertion follows by Proposition \ref{bPBW2} 2) and 3).

We prove these assertions by the downward induction on $m$. In particular, we assume all the assertions if $c_1 = \cdots = c_m = 0$. The base case $m = \ell$ is examined in Example \ref{one-root} (since $\gamma ^{(\ell)} _{\mathbf i}$ is a simple root).

We set $w_m := s_{i_{m-1}} s_{i_{m-2}} \cdots s_{i_1}$. Put $\beta_1 := w_m \beta - c _{m} \alpha_{i_m}$ and $\beta_1' := w_m \beta' - c' _{m} \alpha_{i_m}$. By the $(m-1)$-times repeated application of the construction $\mathbf i \mapsto \mathbf i^{\sharp}$, we obtain
$$\mathbf i^{\flat} := ( i_{m},\ldots,i_{\ell},i_{1}',i_2',\ldots,i_{m-1}') \in I ^{\ell}.$$

Let $\mathbf d$ and $\mathbf d'$ be the $\mathbf i^{\flat}$-Lusztig data given by $d_1 = c_{m},d_2 = c_{m+1},\ldots, d_{\ell-m+1} = c_{\ell}, d_{\ell-m+2} = \cdots = d_{\ell} = 0$ and $d'_1 = c'_{m},d'_2 = c'_{m+1},\ldots, d'_{\ell-m+1} = c'_{\ell}, d'_{\ell-m+2} = \cdots = d'_{\ell} = 0$, respectively. We also set $\mathbf d[1]$ and $\mathbf d'[1]$ as sequences defined as: $d_j [1] = 0$ ($j=1$) or $d_j$ $(1 < j \le \ell)$, and $d'_j [1] = 0$ ($j=1$) or $d'_j$ $(1 < j \le \ell)$. We have
$$
\widetilde{E} ^{\mathbf i} _{\mathbf c} \cong \mathbb T_{i_1} \cdots \mathbb T_{i_{m-1}} \widetilde{E} ^{\mathbf i^{\flat}} _{\mathbf d} \hskip 3mm \text{ and } \hskip 3mm E ^{\mathbf i} _{\mathbf c'} \cong \mathbb T_{i_1} \cdots \mathbb T_{i_{m-1}} E ^{\mathbf i^{\flat}} _{\mathbf d'}.
$$

\begin{claim}\label{transPBW}
We have
$$
\widetilde{E} ^{\mathbf i^{\flat}} _{\mathbf d} \cong \mathbb T_{i_{m-1}} ^* \cdots \mathbb T_{i_1}^* \widetilde{E} ^{\mathbf i} _{\mathbf c} \hskip 3mm \text{ and } \hskip 3mm E ^{\mathbf i^{\flat}} _{\mathbf d'} \cong \mathbb T_{i_{m-1}}^* \cdots \mathbb T_{i_1}^* E ^{\mathbf i} _{\mathbf c'}.
$$
\end{claim}

\begin{proof}
The assertion for $E ^{\mathbf i} _{\mathbf c'}$ follows if every irreducible constituent of $E ^{\mathbf i} _{\mathbf c'}$ does not vanish by applying $\mathbb T_{i_{m-1}} ^* \cdots \mathbb T_{i_1}^*$. This is guaranteed by Proposition \ref{bPBW2} 2). The case of $\widetilde{E} ^{\mathbf i} _{\mathbf c'}$ follows from Proposition \ref{bPBW2} 3) in addition to the case of $E ^{\mathbf i} _{\mathbf c'}$.
\end{proof}

\begin{claim}\label{vanEQ}
The vanishing of $\mathrm{ext} ^{\bullet} _{R_{\beta}} ( \widetilde{E} ^{\mathbf i} _{\mathbf c}, \widetilde{E} ^{\mathbf i} _{\mathbf c'} )$ follows from the vanishing of
\begin{equation}
\mathrm{ext} ^{\bullet} _{R_{\beta}} ( \widetilde{E} ^{\mathbf i} _{\mathbf c}, E ^{\mathbf i} _{\mathbf c'} ) \cong \mathrm{ext} ^{\bullet} _{R_{c_m \alpha _{i_m}} \boxtimes R_{\beta_1}} ( P_{c_m i_m} \boxtimes \widetilde{E} ^{\mathbf i^{\flat}} _{\mathbf d[1]}, E ^{\mathbf i^{\flat}} _{\mathbf d'} ).\label{1st-ext}
\end{equation}
\end{claim}

\begin{proof}
Thanks to Claim \ref{transPBW}, a repeated use of Theorem \ref{SRF} 4) implies a sequence of isomorphisms
\begin{align*}
\mathrm{ext} ^{\bullet} _{R_{\beta}} ( \widetilde{E} ^{\mathbf i} _{\mathbf c}, E ^{\mathbf i} _{\mathbf c'} ) & \cong \mathrm{ext} ^{\bullet} _{R_{\beta}} ( \mathbb T_1 \cdots \mathbb T_{m-1} \widetilde{E} ^{\mathbf i^{\flat}} _{\mathbf d},\mathbb T_1 \cdots \mathbb T_{m-1} E ^{\mathbf i^{\flat}} _{\mathbf d'} )\\
& \cong \mathrm{ext} ^{\bullet} _{R_{s_{i_1} \beta}} ( \mathbb T_2 \cdots \mathbb T_{m-1} \widetilde{E} ^{\mathbf i^{\flat}} _{\mathbf d}, \mathbb T_2 \cdots \mathbb T_{m-1} E ^{\mathbf i^{\flat}} _{\mathbf d'} )\\
\cdots & \cong \mathrm{ext} ^{\bullet} _{R_{w_m \beta}} ( \widetilde{E} ^{\mathbf i^{\flat}} _{\mathbf d}, E ^{\mathbf i^{\flat}} _{\mathbf d'} ).
\end{align*}
Since $\star$ preserves the projectivity (see e.g. \cite{KL} 2.16) and $\widetilde{E} ^{\mathbf i^{\flat}} _{\mathbf d} \cong P_{c_m i_m} \star \widetilde{E} ^{\mathbf i^{\flat}} _{\mathbf d[1]}$, we conclude an isomorphism
$$\mathrm{ext} ^{\bullet} _{R_{w_m \beta}} ( \widetilde{E} ^{\mathbf i^{\flat}} _{\mathbf d}, E ^{\mathbf i^{\flat}} _{\mathbf d'} ) \cong \mathrm{ext} ^{\bullet} _{R _{c_m i _m} \boxtimes R_{\beta_1}} ( P_{c_m i_m} \boxtimes \widetilde{E} ^{\mathbf i^{\flat}} _{\mathbf d[1]}, E ^{\mathbf i^{\flat}} _{\mathbf d'} ).$$

It remains to deduce $\mathrm{ext} _{R_{\beta}} ^{\bullet} ( \widetilde{E} ^{\mathbf i} _{\mathbf c}, \widetilde{E} ^{\mathbf i} _{\mathbf c'} ) = \{ 0 \}$ from the vanishing of (\ref{1st-ext}).

Since $R_{\beta}$ is a Noetherian ring with finite global dimension, we have a projective resolution $\mathcal P ^{\bullet}$ of $\widetilde{E} ^{\mathbf i} _{\mathbf c}$, which consists of finitely many finitely generated projective $R_{\beta}$-modules. In particular, there exists $x \in \mathbb Z$ so that the degrees of simple quotients of all $R_{\beta}$-module direct summands of $\mathcal P ^{\bullet}$ are $\le x$.

For each $j \in \mathbb Z$, we have a (surjective) $A$-module quotient $\varphi_j : \widetilde{E} ^{\mathbf i} _{\mathbf c'} \rightarrow E_j$ so that {\bf a)} $\ker \, \varphi_j$ is concentrated in degree $> j+x$, and {\bf b)} $E_j$ is a finite successive self-extension of (grading shifts) of $E ^{\mathbf i} _{\mathbf c'}$ by Proposition \ref{bPBW2} 3). Then, $\mathrm{ext} _{R_{\beta}} ^{\bullet} ( \widetilde{E} ^{\mathbf i} _{\mathbf c}, E ^{\mathbf i} _{\mathbf c'} ) = \{ 0 \}$ implies
$$\mathrm{ext} _{R_{\beta}} ^{\bullet} ( \widetilde{E} ^{\mathbf i} _{\mathbf c}, \ker \, \varphi_j )^j = \{ 0 \} = \mathrm{ext} _{R_{\beta}} ^{\bullet} ( \widetilde{E} ^{\mathbf i} _{\mathbf c}, E _j ).$$
This yields $\mathrm{ext} _{R_{\beta}} ^{\bullet} ( \widetilde{E} ^{\mathbf i} _{\mathbf c}, \widetilde{E} ^{\mathbf i} _{\mathbf c'} )^j = \{ 0 \}$ (for each $j$) as required.
\end{proof}

We return to the proof of Theorem \ref{main}.

We have the following two short exact sequences:
\begin{align*}
0 \to L_{c_m i_m} \boxtimes E ^{\mathbf i^{\flat}} _{\mathbf d[1]} \to E ^{\mathbf i^{\flat}} _{\mathbf d} \to C \to 0 & \hskip 5mm \text{as $R_{c_m \alpha_{i_m}} \boxtimes R_{\beta_1}$-modules, and}\\
0 \to L_{c'_m i_m} \boxtimes E ^{\mathbf i^{\flat}} _{\mathbf d'[1]} \to E ^{\mathbf i^{\flat}} _{\mathbf d'} \to C' \to 0 & \hskip 5mm \text{as $R_{c_m' \alpha_{i_m}} \boxtimes R_{\beta_1'}$-modules.}
\end{align*}

\begin{claim}\label{weight-est}
We have $e_{i_m} ( c_m ) C = \{ 0 \}$ and $e_{i_m} ( c'_m ) C' = \{ 0 \}$.
\end{claim}

\begin{proof}
Let $\Psi$ be the set of $\mathbf m \in Y ^{w_m \beta}$ so that
$$e ( \mathbf m ) ( L_{c_m' i_m} \boxtimes E ^{\mathbf i^{\flat}} _{\mathbf d'[1]} ) \neq \{ 0 \}.$$

Let $n = \mathsf{ht} \, w_m \beta$, and let $\mathfrak S$ be the set of minimal length representatives of $\mathfrak S_n / ( \mathfrak S_{c'_m} \times \mathfrak S_{n - c'_m})$ inside $\mathfrak S_{n}$. We set $\mathfrak S^* := \mathfrak S \backslash \{1\}$. We have $e ( \mathbf m ) C' \neq \{ 0 \}$ only if $\mathbf m \in \mathfrak S^* \Psi$. Since $E ^{\mathbf i^{\flat}} _{\mathbf d'[1]}$ belongs to the (essential) image of $\mathbb T_{i_m}$, we have $e _{i_m} ( 1 ) E ^{\mathbf i^{\flat}} _{\mathbf d'[1]} = \{ 0 \}$. On the other hand, we have $Y ^{c_1' \alpha_{i_1}} = \{ (i_1,\ldots,i_1)\}$. Since every element of $\mathfrak S^*$ decreases the number of heading $i_1,\ldots,i_1$, we deduce the assertion for $C'$. The case of $C$ is the same.
\end{proof}

We return to the proof of Theorem \ref{main}. We assume $c_m \ge c'_m$. Applying Claim \ref{weight-est}, we deduce that
$$\mathrm{ext} ^{\bullet} _{R_{c_m \alpha _{i_m}} \boxtimes R_{\beta_1}} ( P_{c_m i_m} \boxtimes \widetilde{E} ^{\mathbf i^{\flat}} _{\mathbf d[1]}, C' ) = \{ 0 \}.$$
Therefore, the vector spaces in (\ref{1st-ext}) are isomorphic to
$$\mathrm{ext} ^{\bullet} _{R_{c_m \alpha _{i_m}} \boxtimes R_{\beta_1}} ( P_{c_m i_m} \boxtimes \widetilde{E} ^{\mathbf i^{\flat}} _{\mathbf d[1]}, L_{c_m i_m} \boxtimes E ^{\mathbf i^{\flat}} _{\mathbf d'[1]} ).$$
Here we have $\mathrm{ext} ^{\bullet}_{R_{c_m \alpha _{i_m}}} ( P_{c_m i_m}, L _{c_m i_m} ) = \mathrm{hom}_{R_{c_m \alpha _{i_m}}} ( P_{c_m i_m}, L _{c_m i_m} ) = \mathbb C$. Therefore, Claim \ref{transPBW} and Theorem \ref{SRF} 4) implies
\begin{align}\nonumber
\mathrm{ext} ^{*} _{R_{c_m \alpha _{i_m}} \boxtimes R_{\beta_1}} ( P_{c_m i_m} \boxtimes \widetilde{E} ^{\mathbf i^{\flat}} _{\mathbf d[1]}, L_{c_m i_m} \boxtimes E ^{\mathbf i^{\flat}} _{\mathbf d'[1]} ) & \cong \mathrm{ext} ^{*} _{R_{\beta_1}} ( \widetilde{E} ^{\mathbf i^{\flat}} _{\mathbf d[1]}, E ^{\mathbf i^{\flat}} _{\mathbf d'[1]} )\\
\cong \mathrm{ext} ^{*} _{R_{s_{i_{m-1}} \beta_1}} ( \mathbb T_{i_{m-1}} \widetilde{E} ^{\mathbf i^{\flat}} _{\mathbf d[1]}, \mathbb T_{i_{m-1}} E ^{\mathbf i^{\flat}} _{\mathbf d'[1]} ) \cong \cdots & \cong \mathrm{ext} ^{*} _{R_{w_m \beta_1}} ( \widetilde{E} ^{\mathbf i} _{\mathbf c[1]}, E ^{\mathbf i} _{\mathbf c'[1]} ),\label{2nd-ext}
\end{align}
where $\mathbf c [1]$ and $\mathbf c'[1]$ are the $\mathbf i$-Lusztig data of $\mathbb T_{i_1} \cdots \mathbb T_{i_{m-1}} \widetilde{E} ^{\mathbf i^{\flat}} _{\mathbf d[1]}$ and  $\mathbb T_{i_1} \cdots \mathbb T_{i_{m-1}} E ^{\mathbf i^{\flat}} _{\mathbf d'[1]}$, respectively. Therefore, we deduce the first two assertions by the induction hypothesis and Claim \ref{vanEQ}.

The fourth assertion follows from the vanishing of (\ref{1st-ext}) and the middle two assertions by applying long exact sequences repeatedly.

Therefore, the induction proceeds and we conclude the result.
\end{proof}

\begin{corollary}\label{PBW-Kos}
Fix a reduced expression $\mathbf i$ and $\beta \in Q^+$. We have
$$\widetilde{E} ^{\mathbf i} _{b} = P _b / \bigl( \sum _{f \in \mathrm{hom} _{R_{\beta}} ( P _{b'}, P_b ), b' <_{\mathbf i} b} \mathrm{Im} \, f \bigr) \hskip 2mm \text{ and } \hskip 2mm E ^{\mathbf i} _{b} = P _b / \bigl( \sum _{f \in \mathrm{hom} _{R_{\beta}} ( P _{b}, \widetilde{E} ^{\mathbf i} _{b} ) ^{>0}} \mathrm{Im} \, f\bigr),$$
where $b$ and $b'$ runs over $B ( \infty )_{\beta}$.
\end{corollary}

\begin{proof}
By Lemma \ref{bPBW} 2), $\widetilde{E} ^{\mathbf i} _{b}$ admits a surjection from $P_b$. By Theorem \ref{main} 3), we conclude that all the simple subquotient $\widetilde{E} ^{\mathbf i} _{b}$ is of the form $L_{b'} \left< k \right>$ for $b \le _{\mathbf i} b'$, and hence the RHS surjects onto $\widetilde{E} ^{\mathbf i} _{b}$. By Theorem \ref{main} 3) and 4), the head of $\ker \, ( P_b \to \widetilde{E} ^{\mathbf i} _{b} )$ must be spanned by $L_{b'} \left< k \right>$ for $b' < _{\mathbf i} b$ and $k \in \mathbb Z$, and hence the both sides are maximal quotients of $P_{b}$ whose simple subquotients are that form. Therefore, they are isomorphic to each other. This proves the first assertion. The second assertion follows by Proposition \ref{bPBW2} 3) and Theorem \ref{main} 2).
\end{proof}

\begin{corollary}\label{orth}
Fix a reduced expression $\mathbf i$ and $\beta \in Q^+$. Then, we have
\begin{align*}
\mathrm{ext} ^{i} _{R_{\beta}} (  \widetilde{E} ^{\mathbf i} _{b}, ( E ^{\mathbf i} _{b'} ) ^* ) = \begin{cases} \mathbb C & (b \neq b', i = 0)\\ \{ 0 \} & (otherwise)\end{cases}, \hskip 2mm \text{ and } \hskip 2mm \left< \widetilde{E} ^{\mathbf i} _{b}, ( E ^{\mathbf i} _{b'} ) ^* \right> _{\mathsf{gEP}} = \delta_{b,b'}
\end{align*}
for every $b,b' \in B ( \infty )_{\beta}$.
\end{corollary}

\begin{proof}
Since the first assertion implies the second assertion, we prove only the first assertion. If $b \le _{\mathbf i} b'$, then the assertion follows from Theorem \ref{main} 4). Thanks to Example \ref{one-root}, each $L_{ci}$ admits a finite resolution by the graded shifts of $P_{ci}$ (for each $c \ge 1$ and $i \in I$). By (the proof of) Proposition \ref{bPBW2} 3), we deduce that each $E_{\mathbf c} ^{\mathbf i}$ admits a finite resolution by the graded shifts by $\widetilde{E}_{\mathbf c} ^{\mathbf i}$. By taking the spectral sequence of this resolution, we deduce
\begin{equation}
\mathrm{ext} ^{\bullet} _{R_{\beta}} ( E ^{\mathbf i} _{b}, ( E ^{\mathbf i} _{b'} ) ^* ) = \{ 0 \}\label{Eexti}
\end{equation}
for each $b < _{\mathbf i} b'$. Since $*$ is an exact functor and $\mathrm{ext} ^i _A$ is a universal $\delta$-functor, we deduce
$$\mathrm{hom} _{R_{\beta}} ( M, N ^* ) \cong \mathrm{hom} _{R_{\beta}} ( N, M ^* )$$
for each $M, N \in R_{\beta} \mathchar`-\mathsf{gmod}$. In particular, we conclude (\ref{Eexti}) unless $b = b'$.

Thanks to Theorem \ref{VV} and the definition of the algebras $A_{(G,\mathfrak X)}$ (and $B_{(G,\mathfrak X)}$) in \S 1, we can replace $R_{\beta}$ with its basic ring to assume that it is non-negatively graded. Then, we have $(E^{\mathbf i} _{b'})^* _k = \{ 0 \}$ for every $k > 0$. For each $j \in \mathbb Z$, we have an $R_{\beta}$-module quotient $\varphi_j : \widetilde{E} ^{\mathbf i} _{b} \rightarrow E_j$ so that {\bf a)} $\ker \, \varphi_j$ is concentrated in degree $> - j$, and {\bf b)} $E_j$ is a finite successive self-extension of (graded shifts) of $E ^{\mathbf i} _{b}$ by Lemma \ref{bPBW2} 3). Then, the minimal projective resolution of $\ker \, \varphi_j$ is concentrated in degree $> -j$. In particular, we have
$$\mathrm{ext} _{R_{\beta}} ^{\bullet} ( \ker \, \varphi_j, (E^{\mathbf i} _{b'})^* )^j = \{ 0 \} = \mathrm{ext} _{R_{\beta}} ^{\bullet} ( E _j, (E^{\mathbf i} _{b'})^* ) \hskip 3mm \text{ for each } \hskip 3mm b\neq b'.$$
This yields $\mathrm{ext} _{R_{\beta}} ^{\bullet} ( \widetilde{E} ^{\mathbf i} _{b},(E^{\mathbf i} _{b'})^* )^j = \{ 0 \}$ (for each $j$) as required.
\end{proof}

\begin{remark}\label{remPBW}
For each $\beta \in Q^+$, the standard normalization
$$\left< P_b, L_b \right> _{\mathsf{gEP}} = \mathsf{gdim} \, \mathrm{hom} _{R_{\beta}} ( P_b, L_b )= \delta _{b,b'},$$
combined with Theorem \ref{VVR}, Corollary \ref{orth} and Theorem \ref{main} implies that $\{ \mathsf{gch} \, \widetilde{E} ^{\mathbf i} _b \}_b$ and $\{ \mathsf{gch} \, E ^{\mathbf i} _b \} _b$ give rise to the lower/upper PBW bases corresponding to $\mathbf i$, respectively.
\end{remark}

\begin{corollary}\label{bPBW3}
For each Lusztig datum $( \mathbf i, \mathbf c )$, we have isomorphism as graded $R_{\mathsf{wt} ( \mathbf i, \mathbf c )}$-modules:
\begin{align*}
\widetilde{E} ^{\mathbf i} _{\mathbf c} & \cong P_{c_1 i_1} \star ( \mathbb T_{i_1} P _{c_2 i_2} ) \star ( \mathbb T_{i_1} \mathbb T_{i_2} P _{c_3 i_3} ) \star \cdots \star ( \mathbb T_{i_1} \cdots \mathbb T_{i_{\ell-1}} P_{c_{\ell} i_{\ell}} )\\
E ^{\mathbf i} _{\mathbf c} & \cong L_{c_1 i_1} \star ( \mathbb T_{i_1} L _{c_2 i_2} ) \star ( \mathbb T_{i_1} \mathbb T_{i_2} L _{c_3 i_3} ) \star \cdots \star ( \mathbb T_{i_1} \cdots \mathbb T_{i_{\ell-1}} L_{c_{\ell} i_{\ell}} ).
\end{align*}
\end{corollary}

\begin{proof}
By Proposition \ref{bPBW2} 1), it remains to compare the characters of the both sides. Thanks to Remark \ref{remPBW}, this follows from the comparison with the definition of the PBW bases in \cite{QG} \S 38 as the definition there yield the graded characters of the RHS.
\end{proof}

\begin{theorem}[Lusztig's conjecture]\label{Lusztig}
For every reduced expression $\mathbf i$ of $w_0$, $\beta \in Q^+$, and $b,b' \in B(\infty)_{\beta}$, we have an equality
$$[P_b : \widetilde{E} ^{\mathbf i} _{b'} ] = [E ^{\mathbf i} _{b'} : L _{b}].$$
In particular, the expansion coefficients of the lower global basis in terms of the lower PBW basis are in $\mathbb N [ t ]$.
\end{theorem}

\begin{remark}
Thanks to \cite{K5} 1.7 (and Lusztig \cite{Lu} 10.6), a projective module $P_b$ admits a filtration by $\{ \widetilde{E} ^{\mathbf i} _{b'} \} _{b'}$ if $\mathbf i$ is adapted to $\Gamma$.
\end{remark}

\begin{proof}[Proof of Theorem \ref{Lusztig}]
By Corollary \ref{orth}, we have
\begin{align*}
\delta_{b, b'} & = \left< P_{b}, L _{b'} ^* \right> _{\mathsf{gEP}} = \sum _{d, d' \in B ( \infty ) _{\beta}} \overline{[P_{b} : \widetilde{E} _{d} ^{\mathbf i}]} \overline{[L_{b'} : E _{d'} ^{\mathbf i}]} \left< \widetilde{E} _{d} ^{\mathbf i}, ( E _{d'} ^{\mathbf i} ) ^* \right> _{\mathsf{gEP}}\\
& = \sum _{d \in B ( \infty ) _{\beta}} \overline{[P_{b} : \widetilde{E} _{d} ^{\mathbf i} ]} \overline{[L_{b'} : E _{d} ^{\mathbf i}]}.
\end{align*}
By applying the bar involution, this shows that
$$( [P_{b} : \widetilde{E} _{d} ^{\mathbf i}] ) ( [E _{d'} ^{\mathbf i} : L _{b'}] )^{-1} = ( \delta_{b, b'}),$$
which is equivalent to the assertion.
\end{proof}

For each $\beta, \beta' \in Q^+$, we define the formal expression $q^{\beta}$ and $q^{\beta'}$ so that $q ^{\beta} \cdot q ^{\beta'} = q ^{\beta + \beta'}$. We define
$$\mathrm{ep} _t ( q ^{\beta} ) := \sum _{n \ge 0} \frac{q^{n \beta}}{(1-t^2)(1-t^4) \cdots (1-t^{2n})} \in \mathbb Q ( t ) [\![Q^+]\!].$$

\begin{corollary}[cf. Problem 2 in Kashiwara \cite{Ka2}]\label{denom}
For each $\beta \in Q ^+$, we set
$$[ P : L ]_{\beta} := ( [ P_b : L_{b'} ] ) _{b,b' \in B ( \infty ) _{\beta}} = ( \left< P_{b'}, P_{b} \right> _{\mathsf{gEP}} ) _{b,b' \in B ( \infty ) _{\beta}}$$
as the square matrix with its determinant $D_{\beta}$. We have
$$\sum _{\beta \in Q^+} D_{\beta} q^{\beta} = \prod _{\alpha \in R^+} \mathrm{ep} _t ( \alpha ).$$
\end{corollary}

\begin{proof}
As in the proof of Theorem \ref{Lusztig}, we factorize
$$[ P : L ]_{\beta} = [ P : \widetilde{E} ]_{\beta}[ \widetilde{E} : E ]_{\beta}[ E : L ]_{\beta},$$
where the second term is the $\# B (\infty)_{\beta}$-square matrix of expansion coefficients between projectives/lower PBWs, lower PBWs/upper PBWs, and upper PBWs/simples, respectively. By Theorem \ref{main} 3), the determinant of the third matrix is $1$. By Theorem \ref{Lusztig}, the determinant of the first matrix is also $1$. By Lemma \ref{bPBW} 2) (cf. Corollary \ref{orth}), we conclude
$$D_{\beta} = \prod _{b \in B ( \infty )_{\beta}} [ \widetilde{E} ^{\mathbf i} _{b} : E ^{\mathbf i} _{b}].$$
By Corollary \ref{bPBW2} and the construction of $\mathbb T_{i_j}$, if we denote $\mathbf c$ the $\mathbf i$-Lusztig datum corresponding to $b$, then we have
$$[ \widetilde{E} ^{\mathbf i} _{b} : E ^{\mathbf i} _{b}] = \prod _{j = 1} ^{\ell} [P_{c_j i_j} : L _{c_j i_j }] = \prod _{j = 1} ^{\ell} \frac{1}{(1-t^2) (1-t^4) \cdots (1-t^{2 c_j})}.$$
This is equivalent to the assertion by a simple counting.
\end{proof}

\begin{remark}
{\bf 1)} By a formal manipulation, we have
$$\left< P_{b}, P_{b'} \right> _{\mathsf{gEP}} = \overline{\left< P_{b}, P_{b'} ^* \right> _{\mathsf{gEP}}} \hskip 5mm \text{ for every } b,b' \in B ( \infty ).$$
Since the RHS calculates the Lusztig inner form $\{ \bullet, \bullet \}$ (\cite{QG} 1.2.10) of the lower global basis, Corollary \ref{denom} yields the Shapovalev determinant formula of quantum groups of type $\mathsf{ADE}$. {\bf 2)} The proof of Corollary \ref{denom} also follows from \cite{K5} 3.12, but here the proof works also from the PBW bases $\{ E ^{\mathbf i} _{b}\}_b$ in which $\mathbf i$ is not an adapted reduced expression of $w_0$.
\end{remark}

\begin{center}
{\bf Differences from the published version}
\end{center}
\begin{itemize}
\item The definition of ``multiple induction" (just after the usual induction in page 10), Lemma 2.13 and Lemma 3.13 are added. 
\item To avoid the error in the proof of the previous version of Lemma 4.2 2), it is deleted from Lemma 4.2.
\item To save the proof, the proof of Proposition 4.6 is modified, and Proposition 4.9 and Corollary 4.15 are added. (The last item recovers the Lemma 4.2 2) in the previous version.)
\item Modified several places in order to adjust with the above changes, or correct minor errors within single statements. 
\end{itemize}
\end{document}